\documentclass[journal]{IEEEtran}

\usepackage{color,amsthm}

\usepackage{cite}
\usepackage{amsmath,amssymb,amsfonts}
\usepackage{graphicx}
\usepackage{textcomp}
\def\BibTeX{{\rm B\kern-.05em{\sc i\kern-.025em b}\kern-.08em T\kern-.1667em\lower.7ex\hbox{E}\kern-.125emX}}

\DeclareMathOperator*{\argmin}{arg\,min}

\usepackage{algorithm}
\usepackage{algorithmic}

\usepackage{enumerate}

\newtheorem{theorem}{Theorem}
\newtheorem{assumption}{Assumption}

\newtheorem{example}{Example}
\newtheorem{lemma}{Lemma}
\newtheorem{corollary}{Corollary}
\newtheorem{remark}{Remark}
\newtheorem{definition}{Definition}
\usepackage{dsfont}

\usepackage{mathtools}

\begin{document}
\title{Robust Model Predictive Control for \\ Nonlinear Systems Using Convex Restriction}
\author{Dongchan Lee, Konstantin Turitsyn, and Jean-Jacques Slotine
\thanks{This work was supported in part by the U.S. Department of Energy Office of Electricity as part of the DOE Grid Modernization Initiative and in part by the National Science Foundation Energy, Power, Control and Networks Award 1809314.}
\thanks{D. Lee and J. Slotine are with the Department of Mechanical Engineering, Massachusetts Institute of Technology, Cambridge, MA 02139, USA (email: dclee@mit.edu, jjs@mit.edu).}
\thanks{K. Turitsyn is with D. E. Shaw Group, New York, NY 10036 (email: turitsyn@mit.edu).}
}

\maketitle

\begin{abstract}
We present an algorithm for robust model predictive control with consideration of uncertainty and safety constraints. Our framework considers a nonlinear dynamical system subject to disturbances from an unknown but bounded uncertainty set. By viewing the system as a fixed point of an operator acting over trajectories, we propose a convex condition on control actions that guarantee safety against the uncertainty set. The proposed condition guarantees that all realizations of the state trajectories satisfy safety constraints. Our algorithm solves a sequence of convex quadratic constrained optimization problems of size $n\cdot N$, where $n$ is the number of states, and $N$ is the prediction horizon in the model predictive control problem. Compared to existing methods, our approach solves convex problems while guaranteeing that all realizations of uncertainty set do not violate safety constraints.
Moreover, we consider the implicit time-discretization of system dynamics to increase the prediction horizon and enhance computational accuracy. Numerical simulations for vehicle navigation demonstrate the effectiveness of our approach.
\end{abstract}


\IEEEpeerreviewmaketitle

\section{Introduction}

Model Predictive Control (MPC) has remained a popular control strategy due to its ability to incorporate complex dynamical systems and safety constraints.
One of MPC's advantages is its elegant formulation for considering safety constraints in safety-critical applications such as navigation, robotics, power systems, and chemical plant regulation. 
Advances in sensing and computation provide new opportunities for MPC formulation to tackle a broader range of systems, where mathematical models can be readily estimated using data. 
However, uncertainties in models or sensors can cause the system to deviate from the planned trajectory, and there is a need to consider robustness in safety-critical applications.

One of the natural ways to guarantee robustness is to construct a tube or a funnel around a nominal trajectory that contains all possible realizations of the state trajectory under disturbances \cite{Camacho2007,kouvaritakis2016model,rawlings2017model,rmpc_morari}.
An example is shown in Figure \ref{fig:intro_plot}, where the grey tube around the robust trajectory represents the bound on possible realizations of trajectories under uncertainty.
If the performance specifications are met within the tube, the controller is verified to be robust against a range of model variations and noise.
One way to achieve tube-based robust MPC is by enforcing stability and time-invariance to the tube, but many of the previous work is limited to linear systems \cite{Kothare1996, Langson2004, Mayne2005a, Mayne2009}. 
Consideration of the nonlinear dynamical system in the presence of uncertainty and safety constraints has remained a challenge.

\begin{figure}[!t]
	\centering 	\includegraphics[width=3.4in]{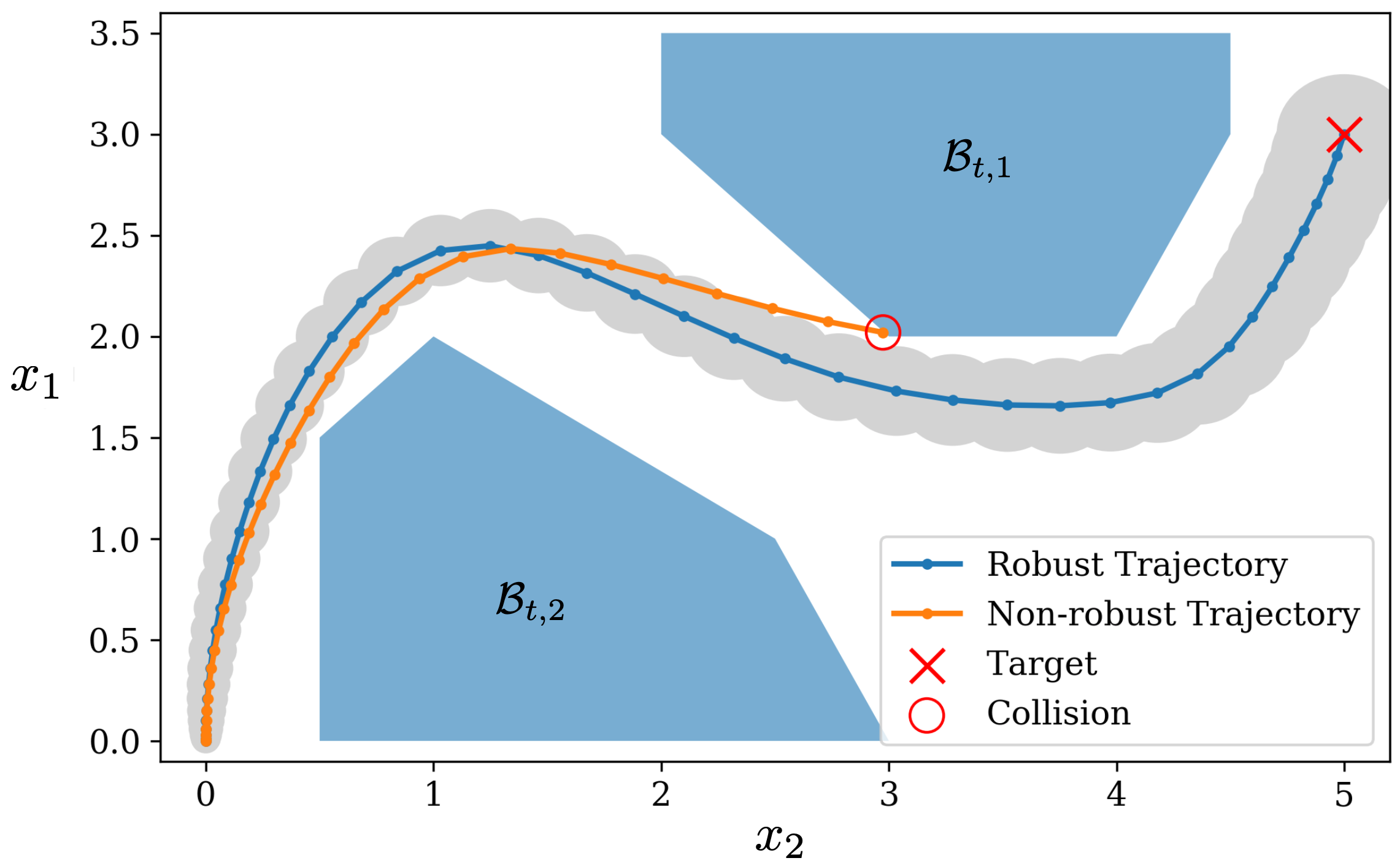}
	\caption{Examples of a non-robust and a robust trajectory are shown. Two blocks, $\mathcal{B}_{t,1}$ and $\mathcal{B}_{t,2}$, are the obstacles that the agent needs to avoid. The non-robust trajectory collides at the red circle while the robust trajectory is able to reach the target point without collision.}
	\label{fig:intro_plot}
\end{figure}

Reachability analysis has been used for the robustness verification of nonlinear systems. The approach computes the reachable set by propagating a set of possible states forward in time using convex relaxation \cite{Lee2017} or using a polynomial approximation of the dynamics \cite{dang12,sassi12}. Interval analysis is also closely related and has been used to bound state estimation error \cite{moore09,jaulin02}.
Another line of related work is a set-theoretic approach \cite{villanueva2015unified, chachuat2015set} and min-max differential equations \cite{villanueva2017robust}.
Alternatively, the backward reachable sets can be computed by Hamilton-Jacobi reachability \cite{Mitchell2005,Chen2018}. However, considering the optimality of control action is not straightforward with reachability analysis.

In order to numerically compute robustness margins, the Sum of Squares (SOS) optimization is often used for polynomial dynamical systems \cite{Parrilo2000}. SOS optimization can be used to search Lyapunov functions \cite{Tobenkin2011,Majumdar2013,Majumdar2017} or contraction metrics \cite{Aylward2008,Singh2017} to verify robustness.
The objective of these optimization problems is to maximize the time-invariant tube using feedback controllers. However, these approaches are restricted to polynomial dynamics and do not consider the optimality of controllers.
In \cite{Majumdar2017,Singh2017,Herbert2018}, the planning of the nominal trajectory to avoid obstacles was considered, but the optimality of the control action was not under consideration.

In this paper, we propose sequential convex restriction for solving the constrained robust MPC problem. The framework can be viewed as an analogous counterpart of Picard-Linderl\"{o}f Theorem for a discrete-time dynamical system. We focus on showing the existence of state trajectory that satisfies safety constraints for time-discretized system dynamics using convex restriction.
The convex restriction was proposed to provide a sufficient condition for solving a system of nonlinear equations subject to inequality constraints \cite{Lee2019b}. The work in \cite{Lee2019b} focused on the applications in the steady-state analysis of electric power grids. These conditions were extended in \cite{Lee2019a} to consider minimization of generation cost while satisfying the nonlinear model of power grids.

The framework is applied to MPC problems by deriving a convex sufficient condition over the control actions such that the resulting state trajectory is robust against the given bounded disturbances. 
Derivation of the sufficient condition involves representing the system with a nonlinear feedback form bounded by an envelope. 
Compared to SOS optimization, our algorithm treats general nonlinear functions and solves convex quadratically constrained quadratic programming (QCQP) problem. QCQP can be solved much efficiently than semidefinite programming, and it can solve the MPC problems with a longer horizon.

The main contributions of the paper can be summarized as follows.
\begin{enumerate}
  \item We provide a convex sufficient condition for robust control actions for the MPC problem and optimize the control and state trajectory using the proposed condition.
  \item We develop methods for both explicit and implicit time-discretization schemes in the MPC problem. We exploit the layered structure of the MPC constraints and Lipschit constant of the dynamics to obtain convex conditions for robustness. 
  \item We demonstrate the effectiveness of the algorithm for a ground vehicle navigation problem with obstacle avoidance. The example includes dynamics that are jointly nonlinear with respect to the state and control variables.
\end{enumerate}

The rest of the paper is organized as follows. In Section II, we present the general formulation of the problem and provide the uncertainty set and the model of the system. Section III provides a guideline for constructing convex restrictions with a provable robust guarantee. Section IV shows the procedures for sequential convex restriction, which uses the derived convex restriction condition. Section V applies the proposed method to a navigation problem. Section VI provides a conclusion.

\section{Problem Statement and Preliminaries}
In this section, we describe the dynamics and the uncertainty model used in the MPC problem. The main advantage of our approach is that we consider general nonlinearity in dynamics with uncertainty as well as nonconvexity in safety constraints.

\subsection{System Dynamics in Continuous time}
We consider a nonlinear dynamical system subject to disturbances and the initial condition $x(0)=w_\textrm{init}$:
\begin{equation}
\frac{d}{dt}x(t)=f(x(t),u(t),w(t)),
\label{eqn:dyn}\end{equation}
where $x(t)\in\mathbb{R}^n$ , $u(t)\in\mathbb{R}^m$, and $w(t)\in\mathbb{R}^r$ represent the state, control, and uncertain variables. The initial condition, $w_\textrm{init}\in\mathbb{R}^n$, and the dynamics parameter, $w(t)$, are subject to uncertainty, which includes external disturbances and model errors. The uncertain variables are assumed to be unknown but bounded by some uncertainty set and will be discussed later. 
The nonlinear dynamics, $f:(\mathbb{R}^n,\mathbb{R}^m,\mathbb{R}^r)\rightarrow \mathbb{R}^n$, is assumed to be Lipschitz continuous according to the following definition.

\begin{assumption}
The nonlinear map $f(x,u,w)$ is uniformly Lipschitz continuous in $x$ and coninuous in $u$ and $w$.
\label{assump:Lipschitz}
\end{assumption}

Given that Assumption \ref{assump:Lipschitz} holds, Picard-Linderl\"{o}f Theorem provides existence and uniqueness of trajectory for the initial value problem in Equation \eqref{eqn:dyn}.

\begin{theorem} (Picard-Linderl\"{o}f Theorem)
Suppose that $f(x,u(t),w(t))$ is uniformly Lipschitz continuous in $x$ and $t$ for some continuous signal $u(t)$ and $w(t)$. Then, there exists a unique solution $x(t)$ to the initial value problem in \eqref{eqn:dyn} for the interval $[t_0-\varepsilon, t_0+\varepsilon]$ with some constant $\varepsilon>0$.
\label{thm:Picard}
\end{theorem}

Picard-Linderl\"{o}f Theorem can be proved by Banach's fixed point theorem to Picard's iteration \cite{Coddington1984}:
\begin{equation}
T_{u,w}[\mathbf{x}](t):=x_0+\int_{t_0}^t f(x(\tau),u(\tau),w(\tau))d\tau.
\label{eqn:Picard}\end{equation}
Our approach will later show an implicitly relation to Picard-Linderl\"{o}f Theorem in that we show safety of our solution by showing the existence of a fixed-point to an operator that acts over trajectories.

\subsection{Time discretization of Continuous Dynamical Systems}
To obtain a numerical solution for the differential equation in~\eqref{eqn:dyn}, the model is converted to a discrete-time dynamical system with a fixed time step $h\in\mathbb{R}$. The explicit and implicit time-discretization schemes are considered in this paper.
\subsubsection{Explicit time-discretization (Forward Euler)}
The explicit scheme approximates the differential equation by
\begin{equation}
x_{t+1}=x_t+h\cdot f(x_t,u_{t+1},w_{t+1}),
\label{eqn:dyn_explicit}\end{equation}
where $x_{t+1}$ can be computed explicitly given $x_t$.

\subsubsection{Implicit time-discretization (Backward Euler)}
The implicit scheme approximates the diffential equation by
\begin{equation}
x_{t+1}=x_t+h\cdot f(x_{t+1},u_{t+1},w_{t+1}),
\label{eqn:dyn_implicit}\end{equation}
where $x_{t+1}$ can be obtained by solving a system of nonlinear equations given $x_t$. Implicit time-discretization schemes admits more accurate solution compared to explicit scheme, and may use larger step size $h$ to predict longer horizon.

\subsection{Modeling Safety Constraints}
The state of the system is constrained by safety constraints, which forms a general nonconvex set denoted by $\mathcal{X}_t$. As an example, these constraints could include physical obstacles that the navigating agent must avoid and safety limits that the system and controller need to respect. We represent the safety constraints in the form of avoiding $s$ convex obstacles at time $t$. The state is declared feasible or safe if $x_t\in\mathcal{X}_t$ or equivalently,
\begin{equation}
x_t\notin\mathcal{B}_{t,(i)},\quad i=1,\ldots,s,
\label{eqn:obstacle}
\end{equation}
where $\mathcal{B}_{t,(i)}\subseteq\mathbb{R}^n,\, i=1,\ldots,s$ are convex sets representing the obstacles. 
The subscript $t$ denotes that the obstacle may be time-dependent to represent moving obstacles. The subscript $(i)$ denotes the index of the obstacle.
The safety constraint can be represented as an intersection of the complement of convex obstacles such that
\begin{equation}
\mathcal{X}_t=\left(\bigcup_{i=1}^s\mathcal{B}_{t,(i)}\right)^C=\bigcap_{i=1}^s\mathcal{B}_{t,(i)}^C,
\end{equation}
where $\mathcal{B}^C_{t,(i)}$ denotes the complement of the set $\mathcal{B}_{t,(i)}$.
The safety constraints are assumed to be represented with a finite number of obstacles. This representation includes the majority of practically relevant applications such as the ground vehicle navigation problem.

\subsection{Modeling Uncertainty Set}\label{sec:uncertainty}
There are two main sources of uncertainty considered in this paper.
\begin{itemize}
\item Uncertainty in initial condition: Initial condition for the state is denoted by $w_\textrm{init}$ and is assumed to be unknown but bounded by $\mathcal{W}_\textrm{init}$.
\item Uncertainty in dynamics: Disturbances are denoted by $w_t$ and are assumed to be unknown but bounded by $\mathcal{W}_t$ for time step, $t$.
\end{itemize}
The uncertainty sets are modeled by convex sets with margin, $\gamma$, representing the size of the uncertainty set. We provide two types of uncertainty sets as examples:
\begin{equation}\begin{aligned}
&\mathcal{W}^Q(\gamma)=\{w\mid (w-w^{(0)})^T \Sigma^{-1}(w-w^{(0)}) \leq \gamma^2\}, \\ 
&\mathcal{W}^I(\gamma)=\{w\mid \lvert w_i-w_i^{(0)}\rvert\leq \gamma_i,\; i=1,\ldots,r \},
\end{aligned}\label{eqn:uncertainty}\end{equation}
where the superscripts $Q$ and $I$ denote ellipsoidal and interval uncertainty sets, respectively.
The ellipsoidal uncertainty set, $\mathcal{W}^Q(\gamma)$, has its center at the nominal value, $w^{(0)}\in\mathbb{R}^r$, with variance and radius of $\Sigma\in\mathbb{R}^{r\times r}$ and $\gamma\in\mathbb{R}$, respectively. 
The interval uncertainty set, $\mathcal{W}^I(\gamma)$, is upper and lower bounded by $w_i^{(0)}+\gamma_i$ and $w_i^{(0)}-\gamma_i$ for each element of $w_i$. 
It is possible to extend the analysis to other uncertainty sets, but we limit ourselves to these uncertainty sets for simplicity of presentation.
Given the uncertainty set, the robustness of the trajectory is defined as the state trajectory satisfying the safety constraints under the uncertainty.

\begin{definition}
The control action, $u_t\in\mathcal{U}_t$ for $t=1,\ldots,N$ is a \textit{robust} control trajectory if the state trajectory satisfies the safety constraints for all realizations of uncertain variables. That is, $\forall w_\textrm{init}\in\mathcal{W}_\textrm{init}$ and $\forall w_t\in\mathcal{W}_t$, $x_t\in\mathcal{X}_t$ for all $t$.
\label{def:robust}\end{definition}

Given the definition of a robust control trajectory, the MPC problem aims to minimize the cost associated with the state and control effort.

\subsection{Constrained Robust Model Predictive Control Problem}

In this section, we provide an overview of the MPC problem with safety and robustness constraints.
The robust MPC problem solves the following optimization problem over a finite horizon of $N$ time steps:

\begin{equation}
    \begin{aligned}
        \underset{u,x}{\text{minimize}} \hskip 1em & c(x,u) \\
        \text{subject to} \hskip 1em & \forall w_t\in\mathcal{W}_t,\ \forall w_\textrm{init}\in\mathcal{W}_\textrm{init},  \\
        & x_0=w_\textrm{init}, \\
        & \text{for} \ \ t=0,...,N-1, \\
        & \hskip 2em x_{t+1}\in\mathcal{X}_{t+1},\ u_{t}\in\mathcal{U}_{t}, \\
        \text{(if explicit)} \hskip -1em & \hskip 2em x_{t+1}=x_t+h\cdot (f(x_t,u_t)+w_t), \\
        \text{(if implicit)} \hskip -1em & \hskip 2em x_{t+1}=x_t+h\cdot (f(x_{t+1},u_t)+w_t).
    \end{aligned}
    \label{eqn:CRMPC}
\end{equation}
The objective function considers the worst-case cost under uncertainty, defined by
\begin{equation}\begin{aligned}
c(x,u)&=\max_{w_t\in\mathcal{W}_t}\left(\sum_{t=0}^{N-1}\left[c_{x,t}(x_t)+c_{u,t}(u_t)\right]+c_{x,N}(x_N)\right),
\end{aligned}\label{eqn:worst_cost}\end{equation}
where $c_{x,t}:\mathbb{R}^n\rightarrow\mathbb{R}$ and $c_{u,t}:\mathbb{R}^m\rightarrow\mathbb{R}$ for $t=1,\ldots,N$ are convex cost functions for states and control actions, respectively. 
The implementation of MPC follows the receding horizon fashion where the first control action of the solution from \eqref{eqn:CRMPC} is applied to the plant, and the remaining computed control actions are discarded. This process is repeated with the new system state set to the initial condition. 
In order to solve the resulting MPC problem, we will propose a local search method for minimizing the cost function while ensuring robustness.

\begin{assumption}
There is some given nominal control and nominal state trajectories, $x_t^{(0)}$ and $u_t^{(0)}$ for $t=1,\ldots,N$, that can be used as the initial condition for problem in \eqref{eqn:CRMPC}.
\label{assump:initial_trj}
\end{assumption}

While Assumption \ref{assump:initial_trj} is often satisfied by using the previous trajectory solution as the initial condition, we have explicitly stated it as an assumption to highlight that our approach highly depends on the initial condition.
The initial trajectory does not necessarily have to be robust, and our solution will enforce robustness while minimizing the associated cost.
These initial state trajectory can be obtained by initializing the control action and simulating the system according to \eqref{eqn:dyn_explicit} or \eqref{eqn:dyn_implicit} according to the choice of discretization scheme.
We denote the state's nominal value, control, and uncertain variable by the superscript $(0)$. Our proposed algorithm will iteratively update the control actions such that it is robustly feasible while decreasing the objective function. The nominal values will be updated later in the algorithm, where the number in the superscript denotes the iteration number.

\section{Convex Restriction of \\ Robust Feasible Control Actions}

We derive a sufficient condition for robust control trajectory using convex restriction in this section.
The derivation involves several steps where the system dynamics are represented as a fixed-point equation. The second step involves enforcing convexity by bounding the nonlinearity with envelopes and using Brouwer's fixed point theorem to verify the existence of the solution.

\subsection{Dynamics as a System of Nonlinear Equations}
We consider the \textit{system trajectories} as a collection of system variables over the prediction horizon $N$. The state, control and uncertainty trajectories will be denoted by $\mathbf{x}\in\mathbb{R}^{(n+1)\cdot N}$, $\mathbf{u}\in\mathbb{R}^{m\cdot N}$, and $\mathbf{w}\in\mathbb{R}^{n+r\cdot N}$ where
\begin{displaymath}
\mathbf{x}=\begin{bmatrix}x_0 \\ \vdots \\ x_N \end{bmatrix}, \ \ 
\mathbf{u}=\begin{bmatrix}u_1 \\ \vdots \\ u_N \end{bmatrix}, \ \ 
\mathbf{w}=\begin{bmatrix}w_\textrm{init} \\ w_1 \\ \vdots \\ w_N \end{bmatrix}.
\end{displaymath}
The uncertain variable, $\mathbf{w}$, includes both the uncertain initial condition, $w_\textrm{init}$, and the uncertain dynamics $w_0,\ldots,w_{N-1}$. 
We will write that $\mathbf{w}\in\mathcal{W}(\gamma)$ if $w_\textrm{init}\in\mathcal{W}_\textrm{init}(\gamma)$ and $w_t\in\mathcal{W}_t(\gamma)$ for $t=1,\ldots,N$.
The cardinality of $\mathbf{x}$ will be denoted by $\lvert\mathbf{x}\rvert$ so that $\mathbf{x}\in\mathbb{R}^{\lvert\mathbf{x}\rvert}$.

The dynamic equations of $N$ time steps can be cast as a system of nonlinear equations by concatenating the equality constraints in \eqref{eqn:CRMPC}. 
This formulation converts the dynamic equation in \eqref{eqn:dyn} to finding a zero of a set of algebraic equations $F(\mathbf{x},\mathbf{u},\mathbf{w})=0$ where $F:(\mathbb{R}^{\lvert\mathbf{x}\rvert},\mathbb{R}^{\lvert\mathbf{u}\rvert},\mathbb{R}^{\lvert\mathbf{w}\rvert})\rightarrow\mathbb{R}^{\lvert\mathbf{x}\rvert}$ defines the dynamics of the system. For example, the dynamic equation in \eqref{eqn:dyn_explicit} can be rearranged to $x_t+h\cdot f(x_t,u_{t+1},w_{t+1})-x_{t+1}=0$, and the initial can be written as $w_\textrm{init}-x_0=0$. Then, the set of equations for explicit time-discretization scheme is given by
\begin{equation}\begin{aligned}
&F_\textrm{explicit}(\mathbf{x},\mathbf{u},\mathbf{w})= \\ 
&\hskip5em
\begin{bmatrix} w_\textrm{init}-x_0 \\
x_0+h\cdot f(x_0,u_1,w_1)-x_1 \\
\vdots \\
x_{N-1}+h\cdot f(x_{N-1},u_{N},w_{N})-x_N
\end{bmatrix},
\end{aligned}\label{eqn:F_dyn_explicit}\end{equation}
and the equations for implicit scheme is given by
\begin{equation}\begin{aligned}
&F_\textrm{Implicit}(\mathbf{x},\mathbf{u},\mathbf{w})= \\
&\hskip5em\begin{bmatrix}
w_\textrm{init}-x_0 \\ 
x_0+h\cdot f(x_1,u_1,w_1)-x_1 \\ 
\vdots \\
x_{N-1}+h\cdot f(x_N,u_N,w_N)-x_N
\end{bmatrix}.
\label{eqn:F_dyn_implicit}\end{aligned}\end{equation}

The number of equations in $F(\mathbf{x},\mathbf{u},\mathbf{w})$ is the same as the number of unknown variables in $\mathbf{x}$, which is the state trajectory given the control and uncertain variables.

\subsection{Fixed-Point Analysis of Discrete-time Dynamical Systems}
Consider the nonlinear equation $F(\mathbf{x},\mathbf{u},\mathbf{w})=0$ defined in \eqref{eqn:F_dyn_explicit} or \eqref{eqn:F_dyn_implicit} depending on the choice of time-discretization scheme. Let $\frac{\partial F}{\partial \mathbf{x}}\bigm|_{(0)}=\frac{\partial F}{\partial \mathbf{x}}\bigm|_{\mathbf{x}=\mathbf{x}^{(0)},\mathbf{u}=\mathbf{u}^{(0)},\mathbf{w}=\mathbf{w}^{(0)}}$ denote the Jacobian of $F$ with respect to $\mathbf{x}$ evaluated at the nominal system trajectory. Note that the variables $\mathbf{x}^{(0)}$, $\mathbf{u}^{(0)}$, and $\mathbf{w}^{(0)}$ denote the nominal trajectories, which was introduced in Assumption \ref{assump:initial_trj}.
The dynamic equation, $F(\mathbf{x},\mathbf{u},\mathbf{w})=0$, can be written as the following fixed-point equation:
\begin{equation}
\mathbf{x}=-\left(\frac{\partial F}{\partial \mathbf{x}}\biggm|_{(0)}^{-1}\right)\boldsymbol{g}(\mathbf{x},\mathbf{u},\mathbf{w})-\left(\frac{\partial F}{\partial \mathbf{x}}\biggm|_{(0)}^{-1}\cdot\;\frac{\partial F}{\partial \mathbf{w}}\biggm|_{(0)}\right)\mathbf{w},
\label{eqn:fxdpt_form} \end{equation}
where $\boldsymbol{g}:(\mathbb{R}^{\lvert\mathbf{x}\rvert},\mathbb{R}^{\lvert\mathbf{u}\rvert},\mathbb{R}^{\lvert\mathbf{w}\rvert})\rightarrow \mathbb{R}^{\lvert\mathbf{x}\rvert}$ is the residual function:
\begin{equation}
\boldsymbol{g}(\mathbf{x},\mathbf{u},\mathbf{w})=F(\mathbf{x},\mathbf{u},\mathbf{w})-\left(\frac{\partial F}{\partial \mathbf{x}}\biggm|_{(0)}\right)\mathbf{x} - \left(\frac{\partial F}{\partial \mathbf{w}}\biggm|_{(0)}\right)\mathbf{w}.
\label{eqn:g_residual}\end{equation}
Closed-form expressions for the Jacobian and the residual function are provided in Appendix \ref{appdx:Jacobian} and \ref{apdx:residual}.

\begin{lemma}
The inverse of the Jacobian, $\frac{\partial F}{\partial \mathbf{x}}\bigm|_{(0)}^{-1}$, exists for both explicit and implicit time-discretization schemes if the step size, $h$, is sufficiently small.
Then, the set of dynamic equations, $F(\mathbf{x},\mathbf{u},\mathbf{w})=0$, is satisfied if and only if \eqref{eqn:fxdpt_form} is satisfied.
\label{lemma:fixed_point}\end{lemma}
\begin{proof}
A closed-form representation of the Jacobian inverse is provided in Lemma \ref{lemma:Jac_explicit} and \ref{lemma:Jac_implicit} in Appendix \ref{appdx:Jacobian}. Substituting the residual function in Equation \eqref{eqn:g_residual} to \eqref{eqn:fxdpt_form}, $0=-\frac{\partial F}{\partial \mathbf{x}}\bigm|_{(0)}^{-1}F(\mathbf{x},\mathbf{u},\mathbf{w})$. Since the Jacobian is invertible, $F(\mathbf{x},\mathbf{u},\mathbf{w})=0$.
\end{proof}

The following Theorem uses Brouwer's fixed point theorem to show the existence of the state trajectory by showing that the Equation \eqref{eqn:fxdpt_form} has a fixed point.
First, the matrices $K\in\mathbb{R}^{\lvert\mathbf{x}\rvert \times\lvert\mathbf{x}\rvert}$ and $R\in\mathbb{R}^{\lvert\mathbf{x}\rvert \times\lvert\mathbf{w}\rvert }$ are defined by
\begin{equation}
K=-\frac{\partial F}{\partial \mathbf{x}}\biggm|_{(0)}^{-1},\hskip1.5em
R=-\frac{\partial F}{\partial \mathbf{x}}\biggm|_{(0)}^{-1}\cdot\;\frac{\partial F}{\partial \mathbf{w}}\biggm|_{(0)}.
\label{eqn:K}\end{equation}

\begin{theorem}
Suppose that $T_{\mathbf{u},\mathbf{w}}:\mathcal{P}\rightarrow\mathcal{P}$ is a continuous map where $\mathcal{P}\subseteq\mathbb{R}^n$ is a nonempty, compact, convex set and
\begin{equation}
T_{\mathbf{u},\mathbf{w}}[\mathbf{x}]=K\boldsymbol{g}(\mathbf{x},\mathbf{u},\mathbf{w})+R\mathbf{w}.
\label{eqn:Brouwer}
\end{equation}
Then, there exists some $\mathbf{x}\in\mathcal{P}$ such that $F(\mathbf{x},\mathbf{u},\mathbf{w})=0$.
\label{thm:Brouwer}
\end{theorem}

The proof of Theorem \ref{thm:Brouwer} is provided in Appendix \ref{apdx:Brouwer}. The proof is analogous to Theorem \ref{thm:Picard} in the sense that the dynamical equation was studied with a fixed-point operator that acts over trajectories. 
Theorem \ref{thm:Picard} uses Picard's iteration for a fixed-point equation and Banach's fixed point theorem for proving the existence of a solution. The fixed-point equation used in Theorem \ref{thm:Brouwer} corresponds to Newton's iteration, and Brouwer's fixed point theorem is used to show the existence of trajectory.
Unlike conventional stability analysis in control, we focus on the existence of trajectory and its bounds under uncertain dynamical systems.
We will establish the existence of the state trajectory that satisfies safety constraints using Theorem \ref{thm:Brouwer} by showing that there exists the set $\mathcal{P}$ with the self-mapping property (i.e., $\forall x\in\mathcal{P}$, $T[\mathbf{x}]\in\mathcal{P}$).



\subsection{Tube around the State Trajectories}
Given the control action, the state trajectory depends on the realized value of the uncertain variable. Since the uncertainty lies in a bounded set, the state variables also form a bounded set over the finite horizon. We will denote the outer approximation of state trajectoriesd by
\begin{equation}
\mathcal{P}(\mathbf{x}^u, \mathbf{x}^\ell)=\left\{\mathbf{x}\Bigm| \mathbf{x}^\ell\leq\mathbf{x} \leq \mathbf{x}^u\right\},
\label{eqn:P_tube}\end{equation}
where $\mathbf{x}^u\in\mathbb{R}^{\lvert\mathbf{x}\rvert}$ and $\mathbf{x}^\ell\in\mathbb{R}^{\lvert\mathbf{x}\rvert}$ provide the upper and lower bounds on the state trajectory. 
Given that $\mathbf{x}^u\geq \mathbf{x}^\ell$, the tube $\mathcal{P}(\mathbf{x}^u, \mathbf{x}^\ell)$ is a non-empty, compact, and convex set that is parametrized by $\mathbf{x}^u,\;\mathbf{x}^\ell\in\mathbb{R}^{\vert\mathbf{x}\rvert}$.
This set will be used as a candidate for the self-mapping set in Theorem \ref{thm:Brouwer}, and the parameter $\mathbf{x}^u$ and $\mathbf{x}^\ell$ will be searched via convex optimization. We will refer the set, $\mathcal{P}(\mathbf{x}^u, \mathbf{x}^\ell)$, to a self-mapping tube.

\subsection{Upper-Convex Lower-Concave Envelopes}
Let the function $\boldsymbol{g}_k(\mathbf{x},\mathbf{u},\mathbf{w})$ be the $k$-th element of the function $\boldsymbol{g}(\mathbf{x},\mathbf{u},\mathbf{w})$.
\begin{definition}
Upper-convex and lower-concave envelopes, $\boldsymbol{g}^u_k(\mathbf{x},\mathbf{u},\mathbf{w})$ and $\boldsymbol{g}^\ell_k(\mathbf{x},\mathbf{u},\mathbf{w})$, are convex and concave functions of $\mathbf{x}$, $\mathbf{u}$ and $\mathbf{w}$, respectively. The envelopes bound the function, $\boldsymbol{g}_{k}(\mathbf{x},\mathbf{u},\mathbf{w})$, for all $\mathbf{x}\in\mathcal{X}$, $\mathbf{u}\in\mathcal{U}$, and $\mathbf{w}\in\mathcal{W}$,
\begin{equation}
\boldsymbol{g}^\ell_k(\mathbf{x},\mathbf{u},\mathbf{w})\leq \boldsymbol{g}_{k}(\mathbf{x},\mathbf{u},\mathbf{w})\leq \boldsymbol{g}^u_k(\mathbf{x},\mathbf{u},\mathbf{w}),
\end{equation}
where $\mathcal{X}$, $\mathcal{U}$, and $\mathcal{W}$ are domain of state, control, and uncertainty trajectories.
\end{definition}

An example of upper-convex and lower-concave envelopes are shown in \ref{fig:conc_env} with solid red lines. The red region enclosed by the envelopes contains the nonlinearity in dynamics.
Given Assumption \ref{assump:Lipschitz}, Corollary \ref{corollary:envelope} in Appendix \ref{apdx:envelopes} shows that the upper-convex and lower-concave always exists, and we provide examples used in numerical simulation section.

\begin{figure}[!htbp]
	\centering 	\includegraphics[width=2.3in]{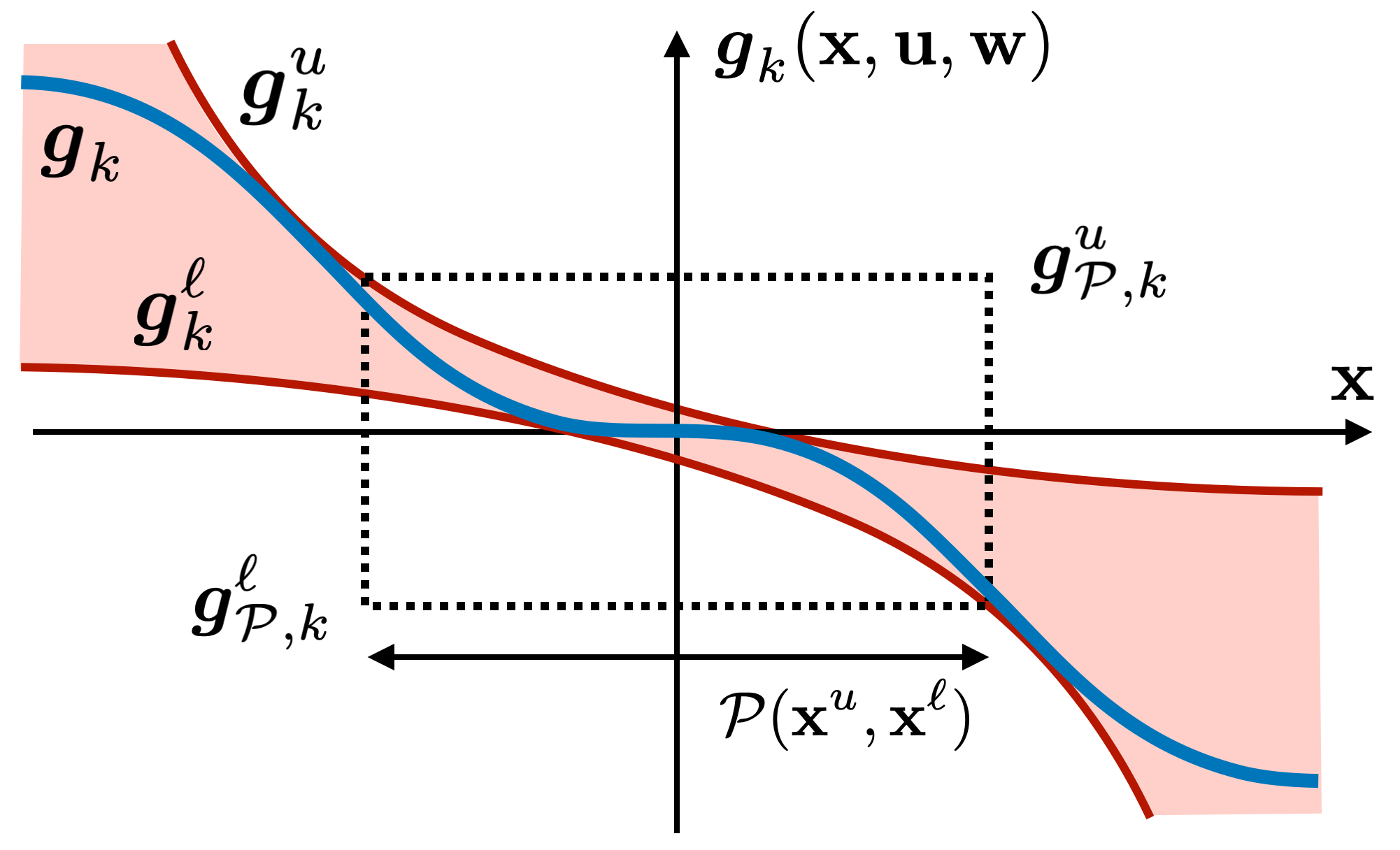}
	\caption{An example of upper-convex and lower-concave envelopes captures the region red that captures the function, $\boldsymbol{g}_k(\mathbf{x},\mathbf{u},\mathbf{w})$.}
	\label{fig:conc_env}
\end{figure}

Using upper-convex and lower concave envelopes, the contribution of nonlinearity within the tube can be bounded as a function of $\mathbf{x}^u$ and $\mathbf{x}^\ell$.
Let us denote the set of vertices of $\mathcal{P}(\mathbf{x}^u,\mathbf{x}^\ell)$ and the outer approximation of $\mathcal{W}(\gamma)$ by
\begin{equation}\begin{aligned}
\partial\mathcal{P}&=\{\mathbf{x}\mid \mathbf{x}_i\in\{\mathbf{x}^u_i,\;\mathbf{x}^u_i\},\, i=1,\ldots,\lvert\mathbf{x}\rvert\}, \\
\partial\mathcal{W}&=\{\mathbf{w}\mid \mathbf{w}_i\in\{\mathbf{w}^\ell_i,\;\mathbf{w}^u_i\},\, i=1,\ldots,\lvert\mathbf{w}\rvert\},
\end{aligned}\label{eqn:partialPW}\end{equation}
where $\mathbf{w}^u_i=\max_{\mathbf{w}\in\mathcal{W}(\gamma)}\mathbf{w}_i$ and $\mathbf{w}^\ell_i=\min_{\mathbf{w}\in\mathcal{W}(\gamma)}\mathbf{w}_i$.

\begin{lemma}
Suppose that for $k=1,\ldots,\lvert\mathbf{x}\rvert$,
\begin{equation}\begin{aligned}
\boldsymbol{g}^u_{\mathcal{P},k}&\geq \boldsymbol{g}^u_k(\mathbf{x},\mathbf{u},\mathbf{w}), &\forall\, \mathbf{x}\in \partial\mathcal{P},\;\forall\, \mathbf{w}\in\partial\mathcal{W},  \\
\boldsymbol{g}^\ell_{\mathcal{P},k}&\leq \boldsymbol{g}^\ell_k(\mathbf{x},\mathbf{u},\mathbf{w}), &\forall\,\mathbf{x}\in \partial\mathcal{P},\;\forall\, \mathbf{w}\in\partial\mathcal{W},
\end{aligned}\label{eqn:g_P}\end{equation}
then for all $\mathbf{x}\in\mathcal{P}(\mathbf{x}^u,\mathbf{x}^\ell)$, $\mathbf{u}\in\mathcal{U}$, and $\mathbf{w}\in\mathcal{W}(\gamma)$,
\begin{equation}
\boldsymbol{g}^\ell_{\mathcal{P},k}\leq \boldsymbol{g}_{k}(\mathbf{x},\mathbf{u},\mathbf{w})\leq \boldsymbol{g}^u_{\mathcal{P},k}.
\end{equation}
\label{lemma:g_P_bound}
\end{lemma}

The proof of Lemma \ref{lemma:g_P_bound} is presented in Appendix \ref{apdx:g_P_bound}, and an illustration is shown in Figure \ref{fig:conc_env}. The variables $\boldsymbol{g}^u_{\mathcal{P},k}$ and $\boldsymbol{g}^\ell_{\mathcal{P},k}$ provide the upper and lower bound on function $\boldsymbol{g}_k$ over the domain, $\mathcal{P}(\mathbf{x}^u, \mathbf{x}^\ell)$. Since the upper envelope is convex, its maximum occurs at one of the vertices of the set $\mathcal{P}(\mathbf{x}^u, \mathbf{x}^\ell)$. Similarly, since the lower envelope is concave, its minimum occurs at one of the vertices of the self-mapping tube.
Notice that the number of inequality constraints in \eqref{eqn:g_P} is $2^{\lvert\mathbf{x}\rvert + \lvert\mathbf{w}\rvert }$ since it lists all possible vertices. However, due to the decoupled nonlinear structure of $\boldsymbol{g}$ between the time steps, the number of vertices that need to be tracked is limited by $N\cdot 2^{n + r}$. Furthermore, Section \ref{sec:scalability} will show that by assuming the system has a sparse nonlinear coupling between its state variables, the number of constraints is bounded by $O(n\cdot N)$.

\subsection{Convex Restriction of Control Actions}
Using envelopes and the fixed-point equation, we present the convex sufficient condition that guarantees that the self-mapping tube in \eqref{eqn:P_tube} is indeed the outer-approximation of possible state trajectories.
Let the constant matrices $K^+,\,K^-\in\mathbb{R}^{\lvert\mathbf{x}\rvert \times \lvert\mathbf{x}\rvert}$ denote $K^+_{ij}=\max\{K_{ij},0\}$ and $K^-_{ij}=\min\{K_{ij},0\}$ for each element of $K$.
The following theorem provides a convex inner-approximation of the control action and an interval outer-approximation of the state trajectory against the uncertainty with a given robustness margin $\gamma$.

\begin{theorem}
Suppose that for a control trajectory $\mathbf{u}$, there exist variables $\mathbf{x}^u,\;\mathbf{x}^\ell,\;\boldsymbol{g}_\mathcal{P}^u,\;\boldsymbol{g}_\mathcal{P}^\ell\in\mathbb{R}^{\lvert\mathbf{x}\rvert}$ that satisfies convex inequality constraints in \eqref{eqn:g_P} and
\begin{equation}\begin{aligned}
&K^+\boldsymbol{g}^u_\mathcal{P}+K^- \boldsymbol{g}^\ell_\mathcal{P}+\xi^u(\gamma)\leq\mathbf{x}^u, \\
&K^+\boldsymbol{g}^\ell_\mathcal{P}+K^- \boldsymbol{g}^u_\mathcal{P}+\xi^\ell(\gamma)\geq\mathbf{x}^\ell,
\end{aligned}\label{eqn:cvxrs_solv}\end{equation}
and for $i=1,\ldots,\lvert\mathbf{x}\rvert$,
\begin{equation}\begin{aligned}
\xi_i^u(\gamma)&=\max_{\mathbf{w}\in\mathcal{W}(\gamma)}R_i\mathbf{w}, \\
\xi_i^\ell(\gamma)&=\min_{\mathbf{w}\in\mathcal{W}(\gamma)}R_i\mathbf{w}.
\end{aligned}\label{eqn:xi}\end{equation}
Then, for every $\mathbf{w}\in\mathcal{W}(\gamma)$, there exists a state trajectory $\mathbf{x}$ such that $\mathbf{x}\in\mathcal{P}(\mathbf{x}^u,\mathbf{x}^\ell)$.
\label{thm:solv}
\end{theorem}

The proof is provided in Appendix \ref{apdx:solv}. The variables $\xi^u(\gamma)$ and $\xi^\ell(\gamma)$ in Equation \eqref{eqn:xi} have analytical solutions for number of types of uncertainty set. The following lemma provides the solutions for ellipsoidal and interval uncertainty sets in \eqref{eqn:uncertainty}.

\begin{lemma}
For an ellipsoidal uncertainty set $\mathcal{W}^Q(\gamma)$ in \eqref{eqn:uncertainty}, $\xi_i^u(\gamma)$ and $\xi_i^\ell(\gamma)$ in \eqref{eqn:xi} are given by
\begin{equation}\begin{aligned}
\xi_i^u(\gamma)&=R_i\mathbf{w}^{(0)}+\gamma\lVert R_i\Sigma^{1/2}\rVert_2, \\
\xi_i^\ell(\gamma)&=R_i\mathbf{w}^{(0)}-\gamma\lVert R_i\Sigma^{1/2}\rVert_2.
\end{aligned}\label{eqn:xi_quad}\end{equation}
for $i=1,\ldots,\lvert\mathbf{x}\rvert$ where $\Sigma\in\mathbb{R}^{\lvert\mathbf{w}\rvert\times \lvert\mathbf{w}\rvert}$ is given by
\begin{equation}\begin{aligned}
\Sigma&=\mathbf{blkdiag}(\Sigma_\textrm{init},\Sigma_0,\ldots,\Sigma_{N-1}).
\end{aligned}\end{equation}

For an interval uncertainty set $\mathcal{W}^I(\gamma)$ in \eqref{eqn:uncertainty}, $\xi_i^u(\gamma)$ and $\xi_i^\ell(\gamma)$ in \eqref{eqn:xi} are given by
\begin{equation}\begin{aligned}
\xi_i^u(\gamma)&=R_i\mathbf{w}^{(0)}+\sum_{k=1}^{\lvert\mathbf{w}\rvert}\lvert R_{ik}\rvert\gamma_k, \\
\xi_i^\ell(\gamma)&=R_i\mathbf{w}^{(0)}-\sum_{k=1}^{\lvert\mathbf{w}\rvert}\lvert R_{ik}\rvert\gamma_k.
\end{aligned}\label{eqn:xi_interval}\end{equation}
\label{lemma:xi}\end{lemma}

The proof of Lemma \ref{lemma:xi} is presented in Appendix \ref{apdx:max_uncertainty}. Next, we add safety constraints to the MPC problem by ensuring the self-mapping tube $\mathcal{P}(\mathbf{x}^u,\,\mathbf{x}^\ell)$ avoids all obstacles.

\subsection{Convex Restriction of Safety Constraints}
In this section, we propose a procedure for deriving the convex restriction of the safety constraints.
The objective is to find a convex subset of $\mathcal{X}_t$ around $x^{(0)}_t$ for $t=1,\ldots,N$. The restricted convex set will be used to certify that the trajectories lie inside the safety constraints.
The procedure relies on the projection of nominal state to obstacles at each time step,
\begin{equation}
P_{\mathcal{B}_{t,i}}[x_t^{(0)}]=\argmin_{x\in\mathcal{B}_{t,i}} \lVert x-x_t^{(0)}\rVert_2^2,
\label{eqn:c_safety}\end{equation}
where $\mathcal{B}_{t,i},\ i=1,\ldots,s$ are obstacles introduced in Section \ref{sec:uncertainty}.
The following lemma provides the convex restriction of the safety constraints using the projections.

\begin{lemma}
The state at time step $t$ satisfies the safety constraints, $x_t\in\mathcal{X}_t$, if $x_t$ satisfies
\begin{equation}
L_tx_t+d_t<0,
\label{eqn:Ldt}
\end{equation}
where the constants $L_t\in\mathbb{R}^{s\times n}$ and $d_t\in\mathbb{R}^s$ are
\begin{displaymath}\begin{aligned}
L_t=\begin{bmatrix} (P_{\mathcal{B}_{t,1}}[x_t^{(0)}]-x_t^{(0)})^T \\ \vdots \\ (P_{\mathcal{B}_{t,s}}[x_t^{(0)}]-x_t^{(0)})^T \end{bmatrix}, \ d_t=-L_t\begin{bmatrix} P_{\mathcal{B}_{t,1}}[x_t^{(0)}] \\ \vdots \\ P_{\mathcal{B}_{t,s}}[x_t^{(0)}] \end{bmatrix}.
\end{aligned}\end{displaymath}
\label{lemma:cvxrs_safety}
\end{lemma}

A formal proof is presented in Appendix \ref{apdx:lemma_safety}, and Figure \ref{fig:cvxrs_safety} illustrates the underlying idea.
Since the obstacles $\mathcal{B}_{t,i}$ are assumed to be convex, the supporting hyperplane at the projection provides a half-space that obstacle cannot exist. By intersecting these half-spaces, we can derive a convex restriction of safety constraints at each state at time $t$. Next, we provide a sufficient condition for the robust feasible control action by ensuring that the self-mapping tube $\mathcal{P}(\mathbf{x}^u,\mathbf{x}^\ell)$ lies inside the convex restriction of safety constraints.

\begin{figure}[!t]
	\centering 	\includegraphics[width=3.2in]{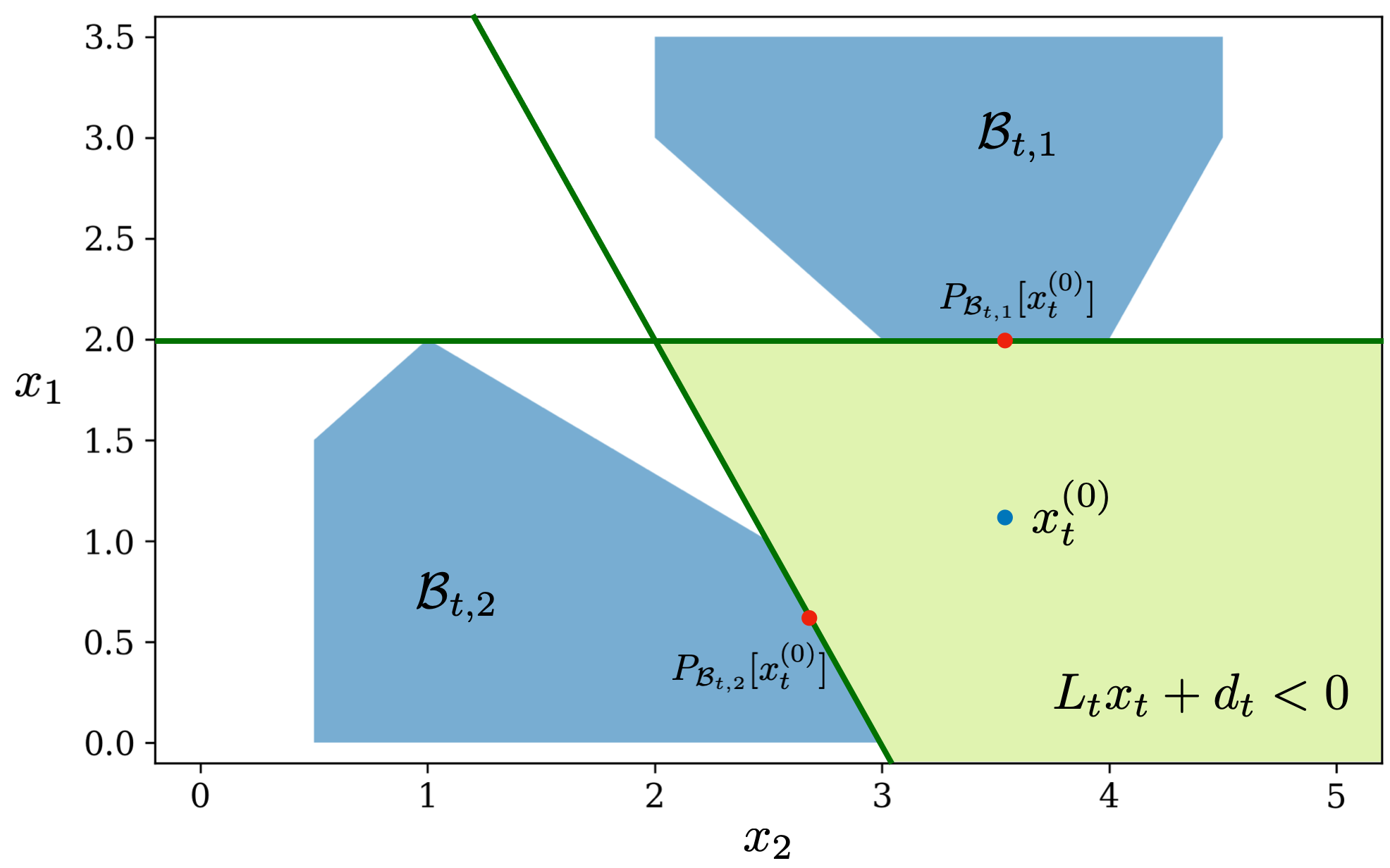}
	\caption{Illustration of convex restriction of safety constraints. The blue dot shows the current nominal state $x_t^{(0)}$, and red dots show its projection to the obstacles. The supporting hyperplanes at the projected point provide half-spaces that are guaranteed to avoid the obstacles. By finding the intersection of half-spaces, convex restriction of safety constraints are shown in the green region.}
	\label{fig:cvxrs_safety}
\end{figure}

\begin{theorem}
The control trajectory $\mathbf{u}$ is a robust feasible control action according to Definition \ref{def:robust} if there exist variables $\mathbf{x}^u,\;\mathbf{x}^\ell,\;\boldsymbol{g}_\mathcal{P}^u,\;\boldsymbol{g}_\mathcal{P}^\ell\in\mathbb{R}^{\lvert\mathbf{x}\rvert}$ that satisfies convex inequality constraints in \eqref{eqn:g_P}, \eqref{eqn:cvxrs_solv}, and
\begin{equation}
L_t^+x_t^u+L_t^-x_t^\ell+d_t<0,\quad \forall t=1,\ldots,N.
\label{eqn:cvx_restr}\end{equation}
\end{theorem}
\begin{proof}
Conditions \eqref{eqn:g_P} and \eqref{eqn:cvxrs_solv} ensure that there exists a state trajectory $\mathbf{x}\in\mathcal{P}(\tilde{\mathbf{z}})$ from Theorem \ref{thm:solv}.
Inequality condition in \eqref{eqn:cvx_restr} ensures that
\begin{displaymath}\begin{aligned}
   L_tx_t+d_t\leq L_t^+x_t^u+L_t^-x_t^\ell+d_t<0,
\end{aligned}\end{displaymath}
 for $t=1,\ldots,N$. From Lemma \ref{lemma:cvxrs_safety}, for all $\mathbf{x}\in\mathcal{P}(\mathbf{x}^\ell, \mathbf{x}^u)$, $x_t\in\mathcal{X}_t$, and thus there exists state trajectories satisfying the safety constraints for all $\mathbf{w}\in\mathcal{W}$. 
\end{proof}


\section{Sequential Convex Restriction}
In this section, we will describe the details of the sequential convex restriction algorithm. The algorithm solves the original constrained robust MPC problem in~\eqref{eqn:CRMPC} via a sequence of deterministic convex optimization problems, which have a well-established theory and tractable algorithms.

\subsection{Objective Function}
The state trajectory is determined by the realizations of uncertain variables and forms a set contained in $\mathcal{P}(\tilde{\mathbf{z}})$. 
\begin{lemma}
If there exists $c^u\in\mathbb{R}$, $c_{\mathcal{P},t}\in\mathbb{R}$ for $t=1,\ldots,N$ such that
\begin{equation}\begin{aligned}
c^u&\geq\sum_{t=0}^{N-1}\left[c_{\mathcal{P},t}+c_{u,t}(u_t)\right]+c_{\mathcal{P},N}, \\
c_{\mathcal{P},t}&\geq c_{x,t}(x_t^u),\hskip1em c_{\mathcal{P},t}\geq c_{x,t}(x_t^\ell),
\end{aligned}\label{eqn:cost_cvx}\end{equation}
for $t=1,\ldots,N,$ then $c^u\geq c(\mathbf{x},\mathbf{u})$ for all $\mathbf{x}\in\mathcal{P}(\mathbf{x}^u,\mathbf{x}^\ell)$.
\label{lemma:cost_bound}
\end{lemma}

The proof of Lemma \ref{lemma:cost_bound} is presented in Appendix \ref{apdx:cost_bound}.
Given the convex restriction of robust feasible control action and the over-estimator of the objective function, we can solve the constrained robust MPC problem in~\eqref{eqn:CRMPC} by replacing the nonconvex constraints with convex restriction,
\begin{equation}
	\begin{aligned}
		\underset{\mathbf{u},\mathbf{x}^u,\mathbf{x}^\ell,\boldsymbol{g}_\mathcal{P}^u,\boldsymbol{g}_\mathcal{P}^\ell}{\text{minimize}} \hskip 1em & \sum_{t=0}^{N-1}\left[c_{\mathcal{P},t}+c_{u,t}(u_t)\right]+c_{\mathcal{P},N} \\
		\text{subject to} \hskip 1em & \text{constraints \eqref{eqn:g_P}, \eqref{eqn:cvxrs_solv}, \eqref{eqn:cvx_restr}, and \eqref{eqn:cost_cvx}}.
	\end{aligned}
	\label{eqn:CVXRS}
\end{equation}

The optimal solution of \eqref{eqn:CVXRS} is always robust feasible, and the solution for $\mathbf{u}$ can be directly used as the control action. 
While the constraints in \eqref{eqn:g_P}, \eqref{eqn:cvxrs_solv}, \eqref{eqn:cvx_restr}, and \eqref{eqn:cost_cvx} ensures that the solution will be robust feasible, it is only a sufficient condition for the robust MPC problem. Therefore, the solution for \eqref{eqn:CVXRS} could be a suboptimal solution if the convex restriction does not span the entire non-convex constraint. In order to achieve lower control cost, sequential convex restriction iteratively solve the problem in \eqref{eqn:CVXRS} while updating the nominal point for constructing convex restriction.

\subsection{Sequential Convex Restriction}
Sequential Convex Restriction (SCR) iteratively solves the optimization problem in \eqref{eqn:CVXRS} to solve the original problem in \eqref{eqn:CRMPC}. The convex restriction is always constructed around the nominal point since coefficients $K$ and $R$ in \eqref{eqn:K} depends on the nominal point. By updating the point of evaluation of the Jacobian, the convex restriction gets updated accordingly.
The algorithm iterates between (a) solving the optimization problem with a convex restriction condition constructed around the nominal trajectory, and (b) setting the solution of the optimization problem as the new nominal trajectory.
The convergence properties were studied in \cite{Lee2019}, and power systems applications were considered in \cite{Lee2019b}.
The algorithm for Model Predictive Control application is described in Algorithm \ref{alg:SCRS} with the termination threshold $\varepsilon>0$.

\begin{algorithm}[!htbp]
\begin{algorithmic}
\STATE \textit{Initialization}: $u^{(0)}=u_0$, $x^{(0)}=x_0$
\WHILE{$\lVert\mathbf{u}^{(k)}-\mathbf{u}^{(k-1)}\rVert>\varepsilon$}
\FOR{$t=1:N$}
\FOR{$i=1:s$}
\STATE $P_{\mathcal{B}_{t,i}}[x_t^{(k)}]=\argmin_{x_t\in\mathcal{B}_{t,i}} \lVert x_t-x_t^{(k)}\rVert_2^2$
\ENDFOR
\STATE Compute $L_t$ and $d_t$ in \eqref{eqn:Ldt}
\ENDFOR
\STATE $\mathbf{u}^{(k+1)}=\argmin_{\mathbf{u}} c^u$ subject to constraints in \eqref{eqn:CVXRS}
\FOR{$t=0:N-1$}
\STATE Solve $x^{(k+1)}_t$ using equations in \eqref{eqn:dyn_explicit} or \eqref{eqn:dyn_implicit}
\ENDFOR
\ENDWHILE
\end{algorithmic}
\caption{Sequential Convex Restriction (SCR) for Constrained Robust MPC}
\label{alg:SCRS}
\end{algorithm}

The sequential convex restriction is a local search method, and its convergence depends on the initializations of the variables. The initialization of the algorithm is denoted by $\mathbf{x}^{(0)}$ and can be obtained by any control strategy.
The main part of the algorithm consists of three main steps. The first is finding the convex restriction of the safety constraint, according to Lemma \ref{lemma:cvxrs_safety}.
The second step is solving the convex optimization problem with convex restriction. As was shown in Remark \ref{remark:scalability}, the optimization solves a convex quadratic problem that scales with $n\cdot N$ and $s\cdot N$ for sparse nonlinear systems.
The third step is retrieving the state trajectory using the new control solution by simply simulating the dynamics.
By leveraging warm-start in convex optimization problems, the sequence of problems can be solved efficiently. 
A related algorithm for the non-robust formulation is Sequential Quadratic Programming \cite{nocedal2006numerical}, which has a wide range of applications in MPC \cite{Hovgaard2013,Hovgaard,Alrifaee2017}. The key advantage of Sequential Convex Restriction is that it enables consideration of robust feasibility against the uncertainty.

\subsection{Robustness of the Planned Trajectory}\label{sec:robustness_of_trj}
Given the nominal control trajectory $\mathbf{u}^{(0)}$, the robustness margin of the control action can be defined as the maximum margin of uncertainty that the system can tolerate. This problem can be solved by a convex optimization problem using convex restriction,
\begin{equation}\begin{aligned}
    \underset{\mathbf{u},\mathbf{x}^u,\mathbf{x}^\ell,\boldsymbol{g}_\mathcal{P}^u,\boldsymbol{g}_\mathcal{P}^\ell,\gamma}{\text{maximize}} \hskip 1em & \gamma \\
    \text{subject to} \hskip 1em & \text{\eqref{eqn:g_P}, \eqref{eqn:cvxrs_solv}, \eqref{eqn:cvx_restr}, and } \mathbf{u}=\mathbf{u}^{(0)}
\end{aligned}\label{eqn:robOPT_gamma}\end{equation}

We note that the convex restriction is only a sufficient condition, and the solution of the problem in \eqref{eqn:robOPT_gamma} is the lower bound of the true robustness margin.

\subsection{Scalability} \label{sec:scalability}
Let us define $\mathcal{I}_k\subseteq\{1,\ldots,n\}$ and $\mathcal{J}_k\subseteq\{1,\ldots,r\}$ to be the subsets of indices of $x$ and $w$ such that the function $g_k$ is dependent on. That is, the subsets, $\{x_i\}_{i\in\mathcal{I}_k}$ and $\{w_i\}_{i\in\mathcal{J}_k}$, are sufficient to evaluate the function $g_k(x,u,w)$:
\begin{displaymath}
g_k(x,u,w)=g_k(\{x_i\}_{i\in\mathcal{I}_k},u,\{w_i\}_{i\in\mathcal{J}_k}).
\end{displaymath}

Note that the function $g_k(x,u,w)$ is a nonlinear residual term, which is a remainder after subtracting the linearized dynamics. Thus, the sets, $\mathcal{I}_k$ and $\mathcal{J}_k$, contains the indicies of variables that have a nonlinear coupling with $x_k$. The cardinality of $\mathcal{I}$ is denoted by $\lvert\mathcal{I}_k\rvert\in\mathbb{N}$, and the worst-case cardianlity is denoted by $\lvert\mathcal{I}\rvert=\max_{k\in\{1,\ldots,n\}}\lvert\mathcal{I}_k\rvert$.
Consider the following example for an illustration.

\begin{example} Suppose that the system dynamics in \eqref{eqn:dyn} is provided by $f(x,u,w)
=[x_3\cos{x_4}+w_1,\; x_3\sin{x_4}+w_2,\; u_1,\; x_3\tan{u_2}]^T$. The residual function $g(x,u,w)$ according to \eqref{eqn:g_residual} is
\begin{equation}\begin{aligned}
g(x,u,w)=& \\
&\hskip-3em h\cdot\begin{bmatrix} x_3\cos{x_4}-x_3\cos{x_4^{(0)}}+x_3^{(0)}\sin{x_4^{(0)}}x_4 \\ 
x_3\sin{x_4}-x_3\sin{x_4^{(0)}}-x_3^{(0)}\cos{x_4^{(0)}}x_4 \\ 
u_1 \\ 
x_3\tan{u_2}-x_3\tan{u^{(0)}_2} \end{bmatrix}.
\end{aligned}\end{equation}
Since $g_1$ is a function of $x_3$ and $x_4$, $\mathcal{I}_1=\{3,4\}$. Similarly, $\mathcal{I}_2=\{3,4\}$, $\mathcal{I}_3=\{\emptyset\}$, and $\mathcal{I}_4=\{3\}$. Moreover, $\mathcal{J}_k=\{\emptyset\}$ for $k=1,\ldots,4$. The worst-case cardianlities are given by $\lvert\mathcal{I}\rvert=2$ and $\lvert\mathcal{J}\rvert=0$.
\label{example:navigation}
\end{example}

The scalability of the convex restriction depends on the cardinality of the nonlinear coupling between variables. Since the function $g$ is a nonlinear residual, the linear coupling between variables does not affect the cardinality. The following remark shows the relationship between the cardinality and the number of constraints in convex restriction.

\begin{remark}
The number of constraints in \eqref{eqn:cvx_restr} is bounded by $O(n\cdot N \cdot 2^{\lvert\mathcal{I}\rvert+\lvert\mathcal{J}\rvert})$.
\label{remark:scalability}\end{remark}

A system with a sparse nonlinear structure has a representation such that $\lvert\mathcal{I}\rvert <<n$ and $\lvert\mathcal{J}\rvert <<n$.
Given that the system has a sparse nonlinear coupling, the number of constraints grows linearly with respect to the size of the MPC problem $n\cdot N$.


\begin{remark}
If the system dynamics in \eqref{eqn:dyn} is given by $f(x,u,w)=Ax+Bu+Cw$ where $A$, $B$, and $C$ are constant matrices. Then the nonlinear residual $g(x,u,w)=Bu$ is only a function of $u$.
\end{remark}

\section{Numerical Results}
This section presents a numerical example that is illustrated on a ground vehicle navigation model. This example contains nonconvex safety constraints, which are obstacles in the context of navigation problems. The numerical experiments were done on 3.3 GHz Intel Core i7 with 16 GB Memory, and the convex optimization problems were implemented with Python with CVXPY with ECOS as the solver \cite{jump}. The MOSEK solver was used to solve the convex Quadratically Constrained Quadratic Programming problems generated by convex restriction. 

\subsection{Ground Vehicle Navigation}
\subsubsection{Ground Vehicle Model}
The dynamics of the ground vehicle is given by
\begin{equation}
\frac{d}{dt}\begin{bmatrix} x_1 \\ x_2 \\ v \\ \theta \end{bmatrix}
=\begin{bmatrix} v\cos{\theta}+w_1 \\ v\sin{\theta}+w_2 \\ u_1+w_3 \\ v\tan{u_2}+w_4 \end{bmatrix}
\end{equation}
where $(x_1,x_2)\in\mathbb{R}^2$ and $\theta$ are the vehicle's position and direction. The variable $v$ is the speed, and $u_1$ and $u_2$ are acceleration and steering velocity. Euler's forward method was used for time discretization with the step size $h=0.05$. 
The degree of sparsity in this system is $\lvert \mathcal{I}\rvert$=2 since the terms $v\cos{\theta}$ and $v\sin{\theta}$ involve $v$ and $\theta$ as the dependent variables.
The safety constraints considered two time-invariant obstacles, which are expressed by a polytope of the form $\mathcal{B}_{(i)}$, $i=1,2$. These obstacles are shown in blue in Figure \ref{fig:gv_trj_bound}.
The control actions were subject to the limits, $u_{t,1}\in[-1,\,1]$ and $u_{t,2}\in[-0.785,\,0.785]$, and the vehicle speed is limited by $x_{t,3}\in[-5.55,\,15.28]$.

The uncertainty in initial condition is set to $\mathcal{W}_\textrm{init}(\gamma_\textrm{init})=\{(x_1,x_2,v,\theta)\mid x_1^2+x_2^2\leq\gamma_{\textrm{init},x}^2,\; \lvert v\rvert\leq \gamma_{\textrm{init},v},\; \lvert \theta\rvert\leq \gamma_{\textrm{init},\theta}\}$ 
where $\gamma_{\textrm{init},x}=0.05$, $\gamma_{\textrm{init},v}=0.05$, and $\gamma_{\textrm{init},\theta}=0.005$.
The uncertainty set in dynamics is set to $\mathcal{W}_t(\gamma_\textrm{dyn})=\{w\mid w_1^2+w_2^2\leq\gamma_{\textrm{dyn},x}^2,\; \lvert w_3\rvert\leq \gamma_{\textrm{dyn},v},\; \lvert w_4\rvert\leq \gamma_{\textrm{dyn},\theta}\}$ where $\gamma_{\textrm{dyn},x}=0.05$, $\gamma_{\textrm{dyn},v}=0.005$, and $\gamma_{\textrm{dyn},\theta}=0.005$.
The cost function for the robust MPC problem was set to $c_{u,t}(u_t)=0.01 u_{t,1}^2 + 0.001 u_{t,2}^2$ and $c_{x,t}(x_t)=(x-x_\textrm{target})^TQ(x-x_\textrm{target})$ where $x_\textrm{target}=(3,5,0,0)$ and $Q=\textrm{diag}([1,1,0.5,0.0001])$.

\subsubsection{Constrained Robust Model Predictive Control}
The constrained robust MPC was solved using sequential convex restriction.
The subproblems were solved with SCR described in Algorithm \ref{alg:SCRS}.

\begin{figure}[!htbp]
	\centering
	\includegraphics[width=3.0in]{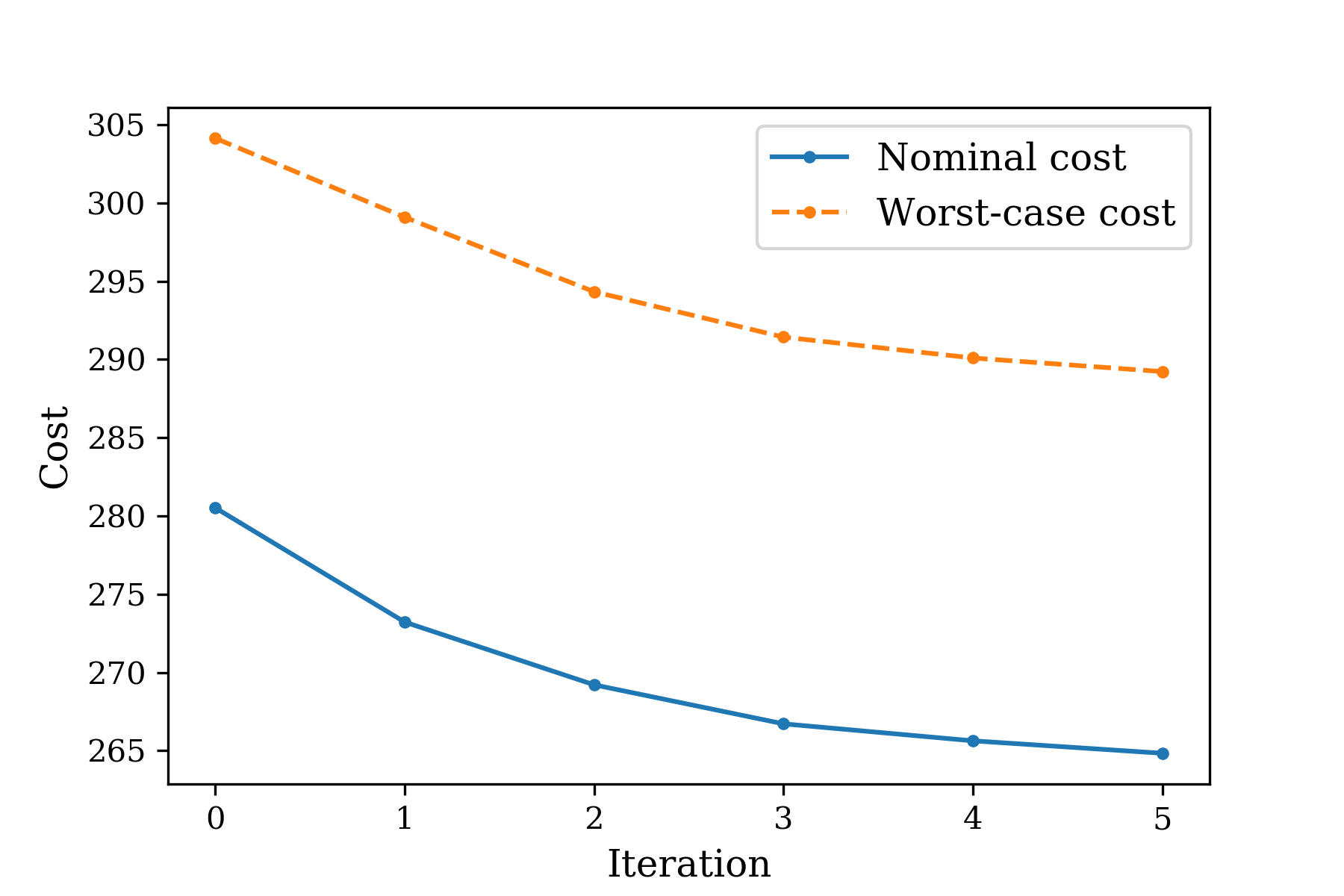}
	\caption{
	The convergence of sequential convex restriction is shown. The blue plot shows the nominal cost as a function of number of iterations, and the orange plot shows the bound on worst-case cost obtained by solving the optimization problem in \eqref{eqn:CVXRS}.}
	\label{fig:convergence}
\end{figure}

Figure \ref{fig:convergence} shows the convergence plot for both nominal and worst-case costs. The average solver time per iteration was 0.733 seconds, and it took five iterations to converge. The figure provides some insight into the conservatism of our approach. The worst-case cost provides the upper bound, and the nominal cost provides the lower bound on the control cost for all realizations of uncertainty set. This corresponds to an uncertainty of 8.44 \% with respect to the worst-case cost.

The following experiments show the result where the initial trajectory was provided by a path-following feedback control. The trajectory was optimized again using SCR with $N=50$.

\begin{figure}[!htbp]
	\centering
	\includegraphics[width=3.3in]{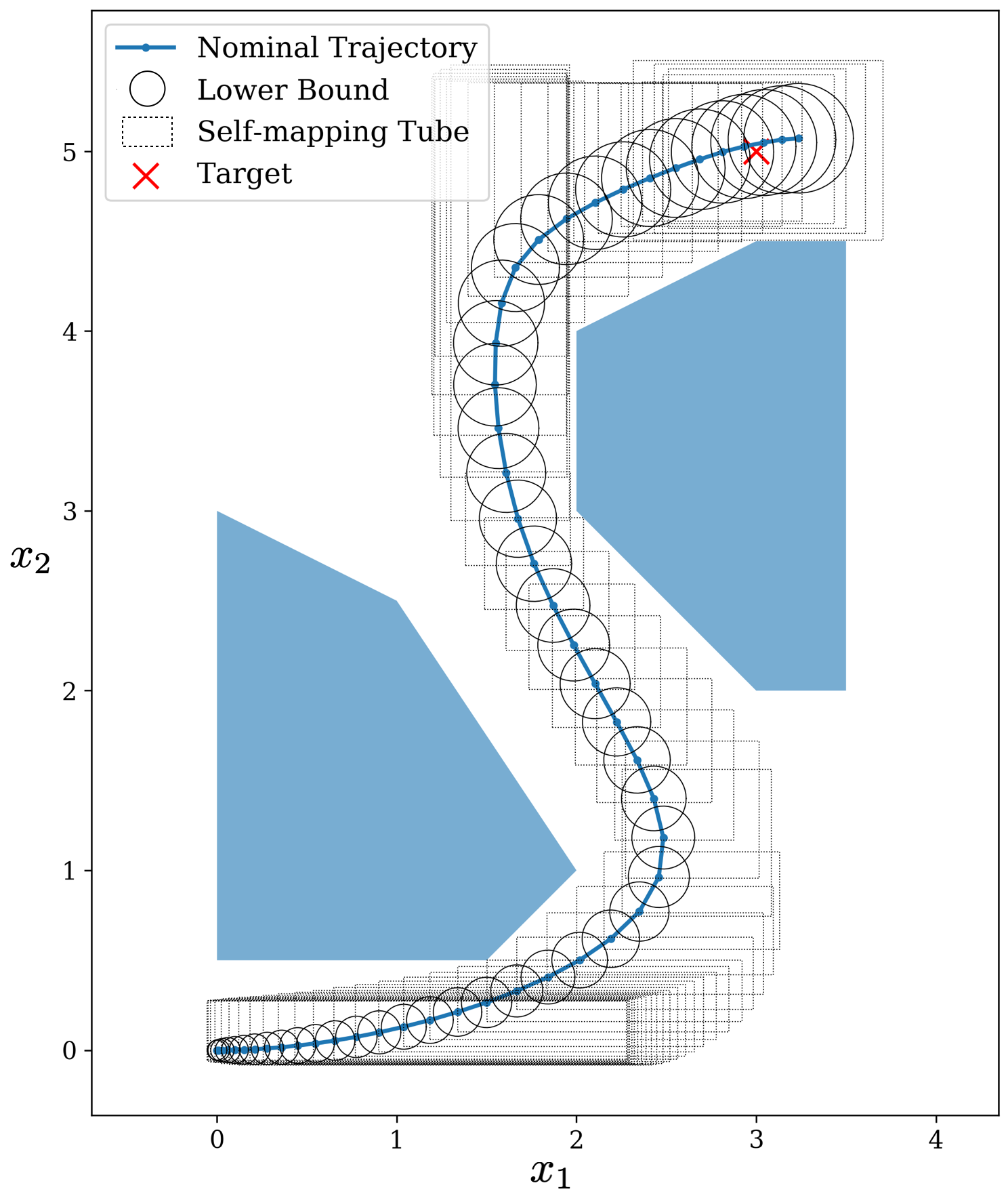}
	\caption{
	The nominal state trajectory obtained by the sequential convex restriction is shown in blue line. The obstacles are shown in two blue regions, and the uncertainty set is shown with black circles. The grey rectangular boxes shows the solution for the self-mapping tube $\mathcal{P}(\tilde{\mathbf{z}})$, which provides the outer approximation of the possible state trajectories.}
	\label{fig:gv_trj_bound}
\end{figure}

Figure \ref{fig:gv_trj_bound} shows the nominal state trajectories of the solution to the algorithm as well as the set of possible state trajectories under uncertainty and the self-mapping tube. The set of possible state trajectories under uncertainty grows with time due to the propagation of uncertain variables in dynamics. The self-mapping tube is guaranteed to contain the possible state trajectory and satisfies the safety constraints. The circles represent the inner approximation of possible state realization. Any point inside the circle can be realized by some uncertainty trajectory within the specified uncertainty set.

\begin{figure}[!htbp]
    \centering
    \includegraphics[width=3.4in]{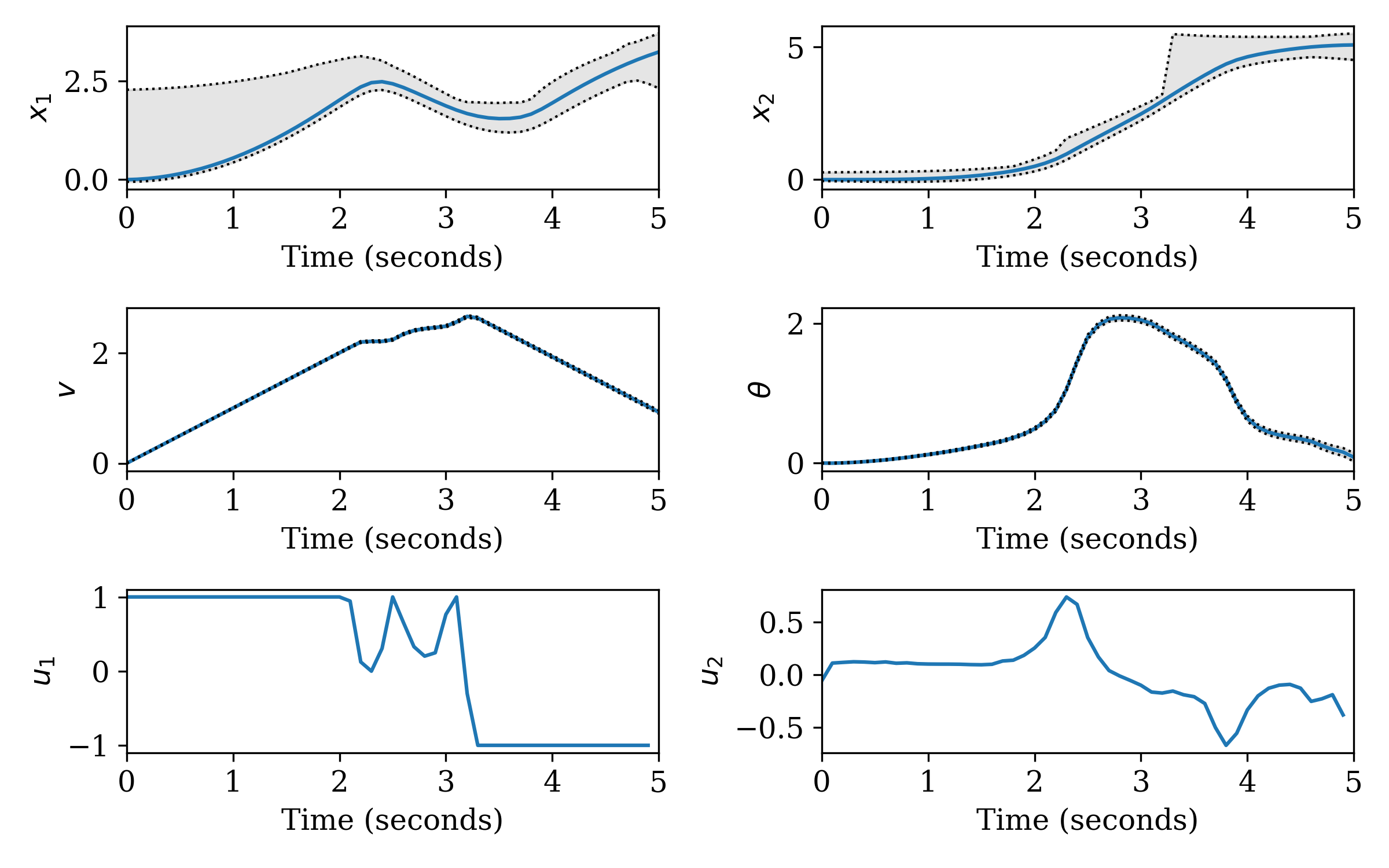}
    \caption{The state trajectories for the obtained control action is shown with the self-mapping tube obtained by the convex restriction. The control action obtained by the convex restriction is shown. The limits of the control inputs are $u_1\in[-100,20]$ and $u_2\in[-1.5,1.5]$.}
    \label{fig:state_trajectory}
\end{figure}

Figure \ref{fig:state_trajectory} shows the state and control trajectories of the solution from sequential convex restriction. The position of the vehicle safely arrives at the target point. The trajectories of the vehicle's velocity and angle are tight to the self-mapping tube since the associated dynamics are linear and are not affected by the uncertain variables.

One unconventional feature in convex restriction that distinguishes itself from conventional approaches is that it does not rely on propagating the uncertainty set in the time domain. 
The outer approximation in Figure \ref{fig:gv_trj_bound} is verified via the fixed-point theorem, and the tube does not become overly conservative over time. 
The outer approximation's tightness is enforced only on the bottom and left faces of the rectangles where the worst-case cost occurs since it is the furthest point away from the target point.
This feature allows us to obtain the tube that satisfies the safety constraint while not overly approximating the worst-case realization of the state trajectories.



\section{Conclusion}
This paper considers constrained robust model predictive control, which requires the satisfaction of safety constraints under all realizations of uncertainty. The MPC formulation's main advantage is that it gives a unified treatment to the problem involving nonlinear dynamics, nonconvex safety constraints, and provable robustness against uncertainty. The numerical simulation showed that the proposed condition is not overly conservative and can be solved efficiently by convex quadratic programming. Furthermore, we considered an implicit time-discretization scheme, which improves the step size and numerical accuracy.

Future work involves generalizing the MPC problem with the Runge-Kutta method for time-discretization. Similar to numerical simulations, an appropriate time-discretization scheme for the MPC problem will require further investigation to improve computational speed and numerical accuracy.
Another line of future work includes extending the analysis to consider secondary feedback control to maximize the system's robustness. The feedback control can establish connections to stability analysis and reduce the variance of the state trajectory.
Combined solutions from stability and convex restriction can further enhance both the optimality and the robustness.

\appendix

\subsection{Jacobian for System of Dynamical Equations}
\label{appdx:Jacobian}
In this section, we derive the Jacobian of the system of equation and its inverse in Equation \eqref{eqn:fxdpt_form}.
\subsubsection{Jacobian for Explicit Time-discretization}
The Jacobian for the system of equations, $F_\textrm{explicit}(\mathbf{x},\mathbf{u},\mathbf{w})$ in \eqref{eqn:F_dyn_explicit}, evaluated at the nominal trajectory is given by
\begin{equation}
\frac{\partial F_\textrm{explicit}}{\partial \mathbf{x}}\biggm|_{(0)}=
\begin{bmatrix} -I & \mathbf{0} & \hdots & \mathbf{0} \\
A_1^{(0)} & -I & \mathbf{0} & \vdots \\
\mathbf{0} & \ddots & \ddots & \mathbf{0} \\
\mathbf{0} & \mathbf{0} & A_{N}^{(0)} & -I \\
\end{bmatrix},
\label{eqn:Jac_F_explicit}
\end{equation}
where $A_t^{(0)}=I+h\cdot\frac{\partial f}{\partial x}\big|_{(x^{(0)}_{t-1},u^{(0)}_t,w^{(0)}_t)}$.

\begin{lemma}
The inverse of the Jacobian in \eqref{eqn:Jac_F_explicit} always exists and has the following closed-form representation,
\begin{equation}
\frac{\partial F_\textrm{explicit}}{\partial \mathbf{x}}\biggm|_{(0)}^{-1}=-\begin{bmatrix} I & \mathbf{0} & \hdots & \mathbf{0} \\
A_{(1,0)}^{(0)}& I & \mathbf{0} & \vdots \\
\vdots & \ddots & \ddots & \mathbf{0} \\
A_{(N,0)}^{(0)} & \hdots & A_{(N,N-1)}^{(0)} & I \\
\end{bmatrix},
\label{eqn:Jac_F_inv_explicit}\end{equation}
where $A_{(i,j)}^{(0)}$ represents the linear sensitivity of the state at time step $i$ with respect to the state at time step $j$. The sensitivity can be solved by applying the chain rule,
\begin{equation}
A_{(i,j)}^{(0)}=\prod_{t=j+1}^{i} \left(I+h\cdot\frac{\partial f}{\partial x}\Big|_{\left(x^{(0)}_{t-1},u^{(0)}_t,w^{(0)}_t\right)}\right).
\end{equation}
\label{lemma:Jac_explicit}
\end{lemma}
Lemma \ref{lemma:Jac_explicit} can be checked by showing that multiplication of \eqref{eqn:Jac_F_explicit} and \eqref{eqn:Jac_F_inv_explicit} is an identity matrix.

\subsubsection{Jacobian for Implicit Time-discretization}

The Jacobian for the system of equations, $F_\textrm{implicit}(\mathbf{x},\mathbf{u},\mathbf{w})$ in \eqref{eqn:F_dyn_implicit}, evaluated at the nominal trajectory is given by
\begin{equation}
\frac{\partial F_\textrm{implicit}}{\partial \mathbf{x}}\biggm|_{(0)}=
\begin{bmatrix} -I & \mathbf{0} & \hdots & \mathbf{0} \\
I & -A_1^{(0)} & \mathbf{0} & \vdots \\
\mathbf{0} & \ddots & \ddots & \mathbf{0} \\
\mathbf{0} & \mathbf{0} & I & -A_N^{(0)} \\
\end{bmatrix},
\label{eqn:Jac_F_implicit}
\end{equation}
where $A_t^{(0)}=I-h\cdot\frac{\partial f}{\partial x}\big|_{(x^{(0)}_t,u^{(0)}_t,w^{(0)}_t)}$.

\begin{lemma}
The inverse of the Jacobian in \eqref{eqn:Jac_F_implicit} always exists and has the following closed-form representation,
\begin{equation}
\frac{\partial F_\textrm{implicit}}{\partial \mathbf{x}}\biggm|_{(0)}^{-1}=-\begin{bmatrix} 
I & \mathbf{0} & \hdots & \mathbf{0} \\
A_{(1,1)}^{(0),-1} & A_{(1,1)}^{(0),-1} & \mathbf{0} & \vdots \\
\vdots & \vdots & \ddots  & \mathbf{0} \\
A_{(N,1)}^{(0),-1} & A_{(N,1)}^{(0),-1} & \hdots & A_{(N,N)}^{(0),-1} \\
\end{bmatrix},
\label{eqn:Jac_F_inv_implicit}
\end{equation}
where $A_{(i,j)}^{(0)}$ represents the linear sensitivity of the state at time step $i$ with respect to the state at time step $j$. The sensitivity can be solved by applying the chain rule,
\begin{equation}
A_{(i,j)}^{(0),-1}=\prod_{t=j}^{i} \left(I-h\cdot\frac{\partial f}{\partial x}\Big|_{\left(x^{(0)}_t,u^{(0)}_t,w^{(0)}_t\right)}\right)^{-1}.
\end{equation}
\label{lemma:Jac_implicit}
\end{lemma}
Similar to explicit scheme, Lemma \ref{lemma:Jac_implicit} can be checked by showing that multiplication of \eqref{eqn:Jac_F_implicit} and \eqref{eqn:Jac_F_inv_implicit} is an identity matrix.


\subsection{Residual Feedback Representation}
\label{apdx:residual}

The residual representation of the system involves the nonlinear residual function $g:(\mathbb{R}^n,\mathbb{R}^m)\rightarrow \mathbb{R}^n$,
\begin{equation}\begin{aligned}
\frac{d}{dt}x&=A(t)x+B(t)w+g(x,u,w) \\
g(x,u,w)&=f(x,u,w)-A(t)x-B(t)w
\end{aligned}\label{eqn:residual}\end{equation}
where $A(t)=\frac{\partial f}{\partial x}\big|_{(x^{(0)},u^{(0)},w^{(0)})}$ and $B(t)=\frac{\partial f}{\partial w}\big|_{(x^{(0)},u^{(0)},w^{(0)})}$ are the Jacobians of the system dynamics evaluated at the nominal values. This representation is related to the Lur'e form in control \cite{Slotine1991,Khalil2002}, which uses the sector bound to contain the nonlinearity. Similarly, we use upper-convex and lower-concave envelopes for bounding the nonlinearity. 

The residual functions in \eqref{eqn:g_residual} has the following representation for explicit scheme,
\begin{equation}
\boldsymbol{g}_\textrm{explicit}(\mathbf{x},\mathbf{u},\mathbf{w})=\begin{bmatrix} \mathbf{0}_n \\ h\cdot g(x_0,u_1,w_1) \\ \vdots \\ h\cdot g(x_{N-1},u_N,w_N) \end{bmatrix},
\end{equation}
and for implicit scheme,
\begin{equation}
\boldsymbol{g}_\textrm{implicit}(\mathbf{x},\mathbf{u},\mathbf{w})=\begin{bmatrix} \mathbf{0}_n \\ h\cdot g(x_1,u_1,w_1) \\ \vdots \\ h\cdot g(x_N,u_N,w_N) \end{bmatrix},
\end{equation}
where $g(x,u,w)$ is the nonlinear residual defined in \eqref{eqn:residual} and $\mathbf{0}_n\in\mathbb{R}^n$ is a vector of zeros. 

The worst-case contribution from the residual function over the domain $x\in\mathcal{P}(x^u,x^\ell)$ and $w\in\mathcal{W}(\gamma)$ are
\begin{equation}\begin{aligned}
\boldsymbol{g}^u_{\mathcal{P},k}(\mathbf{u},\mathbf{x}^u, \mathbf{x}^\ell)&=\max_{\mathbf{x}\in\mathcal{P}(\mathbf{x}^\ell, \mathbf{x}^u)}\max_{\mathbf{w}\in\mathcal{W}(\gamma)}\ \boldsymbol{g}^u_k(\mathbf{x},\mathbf{u},\mathbf{w}), \\
\boldsymbol{g}^\ell_{\mathcal{P},k}(\mathbf{u},\mathbf{x}^u, \mathbf{x}^\ell)&=\min_{\mathbf{x}\in\mathcal{P}(\mathbf{x}^\ell, \mathbf{x}^u)}\min_{\mathbf{w}\in\mathcal{W}(\gamma)}\ \boldsymbol{g}^\ell_k(\mathbf{x},\mathbf{u},\mathbf{w}).
\end{aligned}\end{equation}
These bounds on $\boldsymbol{g}^u_k(\mathbf{x},\mathbf{u},\mathbf{w})$ are established by the conditions in Lemma \ref{lemma:g_P_bound}.

\subsection{Proof of Theorem \ref{thm:Brouwer}}
\label{apdx:Brouwer}

We first introduce Brouwer's fixed point theorem.
\begin{theorem}
(Brouwer's fixed point theorem \cite{brouwer1911abbildung}) Let $\mathcal{P}\subseteq\mathbb{R}^n$ be a nonempty, compact, convex set and $T:\mathcal{P}\rightarrow\mathcal{P}$ be a continuous map. Then, there exists some $x\in\mathcal{P}$ such that $T(x)=x$.
\end{theorem}

The following provides the proof of Theorem \ref{thm:Brouwer}, which extends Brouwer's fixed point theorem to include external parameters, which are control and uncertain variables.

\begin{proof}
Given that there exists a continuous map $\mathcal{P}$ that satisfies the self-mapping property, Brouwer's fixed point theorem guarantees the existence of a fixed point to the equation \eqref{eqn:Brouwer} such that $T_{\mathbf{u},\mathbf{w}}[\mathbf{x}]=\mathbf{x}$ and $x\in\mathcal{P}$. From Lemma \ref{lemma:fixed_point}, the fixed point satisfies $F(\mathbf{x},\mathbf{u},\mathbf{w})=0$.
\end{proof}

\subsection{Upper-Convex Lower-Concave Envelopes}\label{apdx:envelopes}
\begin{lemma}
If $f(x)$ is uniformly Lipschitz in $x$ with Lipschitz constant $L$, then 
\begin{equation}\begin{aligned}
    f(x)\leq f(x^{(0)})+\nabla f(x^{(0)})^T(x-x^{(0)})+\frac{L}{2}\lVert x-x^{(0)}\rVert^2_2, \\
    f(x)\geq f(x^{(0)})+\nabla f(x^{(0)})^T(x-x^{(0)})-\frac{L}{2}\lVert x-x^{(0)}\rVert^2_2. \\
\end{aligned}\label{eqn:proof_lemma9}\end{equation}
\label{lemma:envelope_f}\end{lemma}

\begin{proof}
From calculus, $f(x)=f(y)+\int_0^1\nabla f(z_\alpha)^T(x-y)$ where $z_\alpha=y+\alpha(x-y)$. Then,
\begin{displaymath}\begin{aligned}
&\hskip-3em\lvert f(x)-f(y)-\nabla f(y)^T(x-y)\rvert \\
&= \left\lvert\int_0^1(\nabla f(z_\alpha)-\nabla f(y))^T(x-y)d\alpha\right\rvert \\
&\leq \int_0^1\left\lVert\nabla f(z_\alpha)-\nabla f(y)\right\rVert_2\left\rVert x-y\right\rVert_2 d\alpha \\
&\leq L\int_0^1 \alpha\left\rVert x-y\right\rVert_2^2 d\alpha = \frac{L}{2}\left\rVert x-y\right\rVert_2^2.
\end{aligned}\end{displaymath}
By rearranging and letting $y=x^{(0)}$, we obtain \eqref{eqn:proof_lemma9}.
\end{proof}

\begin{corollary}
Given that the system $f(x,u,w)$ is uniformly Lipschitz with Lipschitz constant $L$, the nonlinearity is bounded by the following envelopes:
\begin{equation}\begin{aligned}
\boldsymbol{g}^u_k(\mathbf{x},\mathbf{u},\mathbf{w})=\boldsymbol{g}^{(0)}_k + \frac{L}{2}\left\lVert\begin{bmatrix}\mathbf{x}-\mathbf{x}^{(0)} \\ \mathbf{u}-\mathbf{u}^{(0)} \\ \mathbf{w}-\mathbf{w}^{(0)}\end{bmatrix}\right\rVert^2_2, \\
\boldsymbol{g}^\ell_k(\mathbf{x},\mathbf{u},\mathbf{w})=\boldsymbol{g}^{(0)}_k - \frac{L}{2}\left\lVert\begin{bmatrix}\mathbf{x}-\mathbf{x}^{(0)} \\ \mathbf{u}-\mathbf{u}^{(0)} \\ \mathbf{w}-\mathbf{w}^{(0)}\end{bmatrix}\right\rVert^2_2.
\end{aligned}\label{eqn:g_bound}\end{equation}
where $\boldsymbol{g}^{(0)}_k = \boldsymbol{g}_k(\mathbf{x}^{(0)},\mathbf{u}^{(0)},\mathbf{w}^{(0)})$ is the residual function evaluated at the nominal trajectory.
\label{corollary:envelope}\end{corollary}
\begin{proof}
From the definition of $\boldsymbol{g}_k(\mathbf{x},\mathbf{u},\mathbf{w})$ in \eqref{eqn:g_residual}, $\nabla\boldsymbol{g}_k(\mathbf{x}^{(0)},\mathbf{u}^{(0)},\mathbf{w}^{(0)})=\mathbf{0}$. Given that the system is uniformly Lipschitz continuous from Assumption \ref{assump:Lipschitz}, the Lipschitz constant is the same as the function $F_k(\mathbf{x},\mathbf{u},\mathbf{w})$. Substituting them to Lemma \ref{lemma:envelope_f}, we obtain \eqref{eqn:g_bound}.
\end{proof}

\subsubsection{Envelopes for Bilinear Function}\label{apdx:bilinear}
Bilinear functions can be bounded by the following envelopes with some constants $\rho_1, \, \rho_2>0$ and the nominal point $x^{(0)},\, y^{(0)}$ \cite{Lee2019b}:
\begin{displaymath}
\begin{aligned}
xy&\geq x^{(0)}y^{(0)}+y^{(0)}(x-x^{(0)})+x^{(0)}(y-y^{(0)}) \\
&\hskip5em -\frac{1}{4}\left[\rho_1(x-x^{(0)})-\frac{1}{\rho_1}(y-y^{(0)})\right]^2, \\
xy&\leq x^{(0)}y^{(0)}+y^{(0)}(x-x^{(0)})+x^{(0)}(y-y^{(0)}) \\
&\hskip5em +\frac{1}{4}\left[\rho_2(x-x^{(0)})+\frac{1}{\rho_2}(y-y^{(0)})\right]^2.
\end{aligned}
\end{displaymath}

\subsubsection{Envelopes for Sine}
Sine and cosine functions have Lipschitz constant of 1. From Lemma \ref{lemma:envelope_f}, sine function with $\theta,\theta^{(0)}\in\mathbb{R}$ can be bounded by the following envelopes:
\begin{displaymath}
\begin{aligned}
\sin\theta&\geq\sin{\theta^{(0)}}+\cos{\theta^{(0)}}(\theta-\theta^{(0)})-\frac{1}{2}(\theta-\theta^{(0)})^2, \\
\sin\theta&\leq \sin{\theta^{(0)}}+\cos{\theta^{(0)}}(\theta-\theta^{(0)})+\frac{1}{2}(\theta-\theta^{(0)})^2.
\end{aligned}
\end{displaymath}

\subsubsection{Envelopes for Cosine}
Similarly, cosine function with $\theta,\theta^{(0)}\in\mathbb{R}$ can be bounded by the following envelopes:
\begin{displaymath}
\begin{aligned}
\cos\theta&\geq\cos{\theta^{(0)}}-\sin{\theta^{(0)}}(\theta-\theta^{(0)})-\frac{1}{2}(\theta-\theta^{(0)})^2, \\
\cos\theta&\leq \cos{\theta^{(0)}}-\sin{\theta^{(0)}}(\theta-\theta^{(0)})+\frac{1}{2}(\theta-\theta^{(0)})^2.
\end{aligned}
\end{displaymath}

\subsubsection{Envelopes for Tangent over a finite domain}
Tangent function over the domain $\theta\in[-\pi/4,\pi/4]$ has Lipschitz constant 4. From Lemma \ref{lemma:envelope_f}, tangent function with $\theta,\theta^{(0)}\in[-\pi/4,\pi/4]$ can be bounded by the following envelopes:
\begin{displaymath}
\begin{aligned}
\tan\theta&\geq\tan{\theta^{(0)}}+\sec^2{\theta^{(0)}}(\theta-\theta^{(0)})-2(\theta-\theta^{(0)})^2, \\
\tan\theta&\leq \tan{\theta^{(0)}}+\sec^2{\theta^{(0)}}(\theta-\theta^{(0)})+2(\theta-\theta^{(0)})^2.
\end{aligned}
\end{displaymath}

\subsection{Proof of Lemma \ref{lemma:g_P_bound}}
\label{apdx:g_P_bound}
\begin{proof}
Since $\partial\mathcal{P}$ is a set of extreme points of $\mathcal{P}$, any point, $\mathbf{x}\in\mathcal{P}$, can be written as a convex combination, $\mathbf{x}=\sum_{\mathbf{x}_i\in\partial\mathcal{P}} \lambda_i\mathbf{x}_i$ where $\lambda_i\geq 0$, $i=1,\ldots,\lvert\mathbf{x}\rvert$ and $\sum_{i=1}^{\lvert\mathbf{x}\rvert}\lambda_i=1$. Since $\boldsymbol{g}^u_k$ is a convex function, for all $\mathbf{w}$,
\begin{equation}\begin{aligned}
&\boldsymbol{g}_k(\mathbf{x},\mathbf{u},\mathbf{w})\leq\boldsymbol{g}^u_k(\mathbf{x},\mathbf{u},\mathbf{w}) \\
&\hskip1em\leq\sum_{\mathbf{x}_i\in\partial\mathcal{P}} \lambda_i\boldsymbol{g}^u_k(\mathbf{x}_i,\mathbf{u},\mathbf{w})\leq\max_{\mathbf{x}_i\in\partial\mathcal{P}}\boldsymbol{g}^u_k(\mathbf{x}_i,\mathbf{u},\mathbf{w}).
\end{aligned}\label{eqn:proof_lemma2}\end{equation}
Similarly, $\partial\mathcal{W}$ contains the extreme points of an outer approximation of $\mathcal{W}$, and for all $\mathbf{x}$, $\boldsymbol{g}_k(\mathbf{x},\mathbf{u},\mathbf{w})\leq\max_{\mathbf{w}_i\in\partial\mathcal{W}}\boldsymbol{g}^u_k(\mathbf{x},\mathbf{u},\mathbf{w}_i)$. 
Condition \eqref{eqn:g_P} ensures that $\boldsymbol{g}^u_{\mathcal{P},k}\geq\max_{\mathbf{w}_i\in\partial\mathcal{W}}\max_{\mathbf{x}_i\in\partial\mathcal{P}}\boldsymbol{g}^u_k(\mathbf{x}_i,\mathbf{u},\mathbf{w}_i)$, and therefore, $\boldsymbol{g}_{k}(\mathbf{x},\mathbf{u},\mathbf{w})\leq \boldsymbol{g}^u_{\mathcal{P},k}$ for all $\mathbf{x}\in\mathcal{P}(\mathbf{x}^u,\mathbf{x}^\ell)$ and $\mathbf{w}\in\mathcal{W}(\gamma)$. Using a similar argument for $\boldsymbol{g}^\ell_k(\mathbf{x},\mathbf{u},\mathbf{w})$, $\boldsymbol{g}_{k}(\mathbf{x},\mathbf{u},\mathbf{w})\geq \boldsymbol{g}^\ell_{\mathcal{P},k}$, which completes our proof.
\end{proof}

\subsection{Proof of Theorem 
\ref{thm:solv}}
\label{apdx:solv}

\begin{proof}
To show that $\mathcal{P}(\mathbf{x}^u,\mathbf{x}^\ell)$ is a self-mapping set with the map, $T_{\mathbf{u},\mathbf{w}}[\mathbf{x}]=K\boldsymbol{g}(\mathbf{x},\mathbf{u},\mathbf{w})+R\mathbf{w}$, substitute $\mathbf{x}=T_{\mathbf{u},\mathbf{w}}[\mathbf{x}]$ to Equation \eqref{eqn:P_tube}. Then, the self-mapping condition requires that for all $\mathbf{x}\in\mathcal{P}(\mathbf{x}^u,\mathbf{x}^\ell)$ and $\mathbf{w}\in\mathcal{W}(\gamma)$, $\mathbf{x}^\ell\leq K\boldsymbol{g}(\mathbf{x},\mathbf{u},\mathbf{w})+R\mathbf{w}\leq\mathbf{x}^u$. Since $\mathcal{P}(\mathbf{x}^u,\mathbf{x}^\ell)$ is an intersection of interval, the self-mapping condition is equivalent to
\begin{displaymath}\begin{aligned}
\max_{\mathbf{w}\in\mathcal{W}(\gamma)} \max_{\mathbf{x}\in\mathcal{P}(\mathbf{x}^u,\mathbf{x}^\ell)} K_i\boldsymbol{g}(\mathbf{x},\mathbf{u},\mathbf{w})+R_i\mathbf{w}\leq\mathbf{x}^u_i \\
\min_{\mathbf{w}\in\mathcal{W}(\gamma)} \min_{\mathbf{x}\in\mathcal{P}(\mathbf{x}^u,\mathbf{x}^\ell)} K_i\boldsymbol{g}(\mathbf{x},\mathbf{u},\mathbf{w})+R_i\mathbf{w}\geq\mathbf{x}^\ell_i
\end{aligned}\end{displaymath}
for $i=1,\ldots,\lvert\mathbf{x}\rvert$.
From condition \eqref{eqn:cvxrs_solv} and Lemma \ref{lemma:g_P_bound}, we can show that
\begin{displaymath}
\begin{aligned}
    &\max_{\mathbf{w}\in\mathcal{W}} \max_{\mathbf{x}\in\mathcal{P}(\mathbf{x}^u,\mathbf{x}^\ell)} (K_i\boldsymbol{g}(\mathbf{x},\mathbf{u},\mathbf{w})+R_i\mathbf{w}) \\
    &\leq\max_{\mathbf{w}\in\mathcal{W}}\max_{\mathbf{x}\in\mathcal{P}(\mathbf{x}^u,\mathbf{x}^\ell)} \left(K_i^+\boldsymbol{g}^u+K_i^-\boldsymbol{g}^\ell\right)(\mathbf{x},\mathbf{u},\mathbf{w})+\max_{\mathbf{w}\in\mathcal{W}}R_i\mathbf{w} \\
    &\leq K_i^+\max_{\substack{\mathbf{x}\in\partial\mathcal{P} \\ \mathbf{w}\in\partial\mathcal{W}}}\boldsymbol{g}^u(\mathbf{x},\mathbf{u},\mathbf{w})+K_i^-\min_{\substack{\mathbf{x}\in\partial\mathcal{P} \\ \mathbf{w}\in\partial\mathcal{W}}}\boldsymbol{g}^\ell(\mathbf{x},\mathbf{u},\mathbf{w})+\xi^u_i(\gamma) \\
    &=K_i^+\boldsymbol{g}^u_\mathcal{P}+K_i^- \boldsymbol{g}^\ell_\mathcal{P}+\xi^u_i(\gamma)\leq \mathbf{x}^u_i,
\end{aligned}
\end{displaymath}
where $\partial\mathcal{P}$ and $\partial\mathcal{W}$ are defined in \eqref{eqn:partialPW}. Therefore, the set $\mathcal{P}(\mathbf{x}^u,\mathbf{x}^\ell)$ satisfies the self-mapping condition in Theorem \ref{thm:Brouwer} for all $\mathbf{w}\in\mathcal{W}(\gamma)$, and there exists $\mathbf{x}\in\mathcal{P}(\mathbf{x}^u,\mathbf{x}^\ell)$ that satisfies $F(\mathbf{x},\mathbf{u},\mathbf{w})$, which corresponds to a state trajectory.
\end{proof}

\subsection{Closed-form Solutions for Optimization Problems over Uncertainty Sets}
\label{apdx:max_uncertainty}

\begin{proof}
For an ellipsoidal uncertainty set $\mathcal{W}^Q(\gamma)$, $\xi_i^u(\gamma)$ in \eqref{eqn:xi} is given by
\begin{displaymath}\begin{aligned}
\xi_i^u(\gamma)&=\max_\mathbf{w}\{R_i\mathbf{w}\mid (\mathbf{w}-\mathbf{w}^{(0)})^T \Sigma^{-1}(\mathbf{w}-\mathbf{w}^{(0)}) \leq \gamma^2\}, \\
&=\max_\mathbf{w}\{R_i\mathbf{w}^{(0)}+\gamma R_i\Sigma^{1/2}\tilde{\mathbf{w}}\mid \lVert\tilde{\mathbf{w}}\rVert_2 \leq 1\}, \\
&=R_i\mathbf{w}^{(0)}+\gamma\lVert R_i\Sigma^{1/2}\rVert_2.
\end{aligned}\end{displaymath}

For interval uncertainty set $\mathcal{W}^I(\gamma)$, $\xi_i^u(\gamma)$ is given by
\begin{displaymath}\begin{aligned}
\xi_i^u(\gamma)&=\max_\mathbf{w}\{R_i\mathbf{w}\mid\lvert \mathbf{w}_i-\mathbf{w}_i^{(0)}\rvert\leq \gamma_i,\; i=1,\ldots,\lvert\mathbf{w}\rvert\}, \\
&=\max_\mathbf{w}\{R_i\mathbf{w}^{(0)}+R_i\tilde{\mathbf{w}}\mid\lvert \tilde{\mathbf{w}}_i\rvert\leq \gamma_i,\; i=1,\ldots,\lvert\mathbf{w}\rvert\}, \\
&=R_i\mathbf{w}^{(0)}+\sum_{k=1}^{\lvert\mathbf{w}\rvert}\lvert R_{ik}\rvert\gamma_k.
\end{aligned}\end{displaymath}
The lower bound $\xi_i^\ell(\gamma)$ can be derived similarly.
\end{proof}

 \subsection{Proof of Lemma 
\ref{lemma:cvxrs_safety}}
\label{apdx:lemma_safety}

\begin{proof}
The necessary and sufficient condition for optimality for the projection problem in \eqref{eqn:c_safety} is
\begin{displaymath}
(P_{\mathcal{B}_{t,i}}[x_t^{(0)}]-x_t^{(0)})^T(\tilde{x}-P_{\mathcal{B}_{t,i}}[x_t^{(0)}])\geq0,\hskip0.5em \forall \tilde{x}\in\mathcal{B}_{t,(i)},
\end{displaymath}
for $i=1,\ldots,s$. The condition, $L_tx_t+d_t<0$, ensures that
\begin{displaymath}
(P_{\mathcal{B}_{t,i}}[x_t^{(0)}]-x_t^{(0)})^T(x_t-P_{\mathcal{B}_{t,i}}[x_t^{(0)}])<0,\ i=1,\ldots,s.
\end{displaymath}
Therefore, $x_t\notin\mathcal{B}_{t,(i)}$ for $i=1,\ldots,s$. Therefore, the state $x_t$ satisfies the safety constraints.
\end{proof}

\subsection{Proof of Lemma \ref{lemma:cost_bound}}
\label{apdx:cost_bound}
\begin{proof}
For all state trajectories $\mathbf{x}\in\mathcal{P}(\mathbf{x}^u,\,\mathbf{x}^\ell)$, the state at each time step can be represented as $x_t = \alpha x_t^u + (1-\alpha) x_t^\ell$ for some scalar $\alpha\in[0,1]$.
Given that the functions $c_{x,t}:\mathbb{R}^n\rightarrow\mathbb{R}$ for $t=1,\ldots,N$ are convex from Equation \eqref{eqn:worst_cost} and the inequalities from \eqref{eqn:cost_cvx}, 
\begin{displaymath}\begin{aligned}
c_{x,t}(x_t) &\leq \alpha c_{x,t}(x_t^u) + (1-\alpha)c_{x,t}(x_t^\ell) \\
&\leq \max \{c_{x,t}(x_t^u), c_{x,t}(x_t^\ell)\} \leq c_{\mathcal{P},t}.
\end{aligned}\end{displaymath}
Then, the conditions in \eqref{eqn:cost_cvx} ensure that
\begin{displaymath}\begin{aligned}
c^u&\geq\sum_{t=0}^{N-1}\left[c_{\mathcal{P},t}+c_{u,t}(u_t)\right]+c_{\mathcal{P},N} \\
&\hskip0em\geq \max_{\mathbf{x}\in\mathcal{P}(\mathbf{x}^u,\,\mathbf{x}^\ell)}\left(\sum_{t=0}^{N-1}\left[c_{x,t}(x_t)+c_{u,t}(u_t)\right]+c_{x,N}(x_N)\right),
\end{aligned}\end{displaymath}
and therefore, $c^u\geq c(\mathbf{x},\mathbf{u})$ for all $\mathbf{x}\in\mathcal{P}(\mathbf{x}^u,\mathbf{x}^\ell)$.
\end{proof}

\bibliographystyle{IEEEtran}
\bibliography{references}

\end{document}